\documentclass[11pt]{tran-l}
\usepackage[a4paper, top=1.5in, bottom=1.5in, left=1in, right=1in]{geometry}

\usepackage{amssymb}

\newtheorem{theorem}{Theorem}[section]
\newtheorem{theoremL}{Theorem}[section]    

\newtheorem{lemma}[theorem]{Lemma}
\newtheorem{prop}[theorem]{Proposition}
\newtheorem{cor}[theorem]{Corollary}

\theoremstyle{definition}
\newtheorem{definition}[theorem]{Definition}

\theoremstyle{remark}
\newtheorem{remark}[theorem]{Remark}

\numberwithin{equation}{section}

\begin{document}

\title[Large fluctuations of random multiplicative functions]{
Large fluctuations of 
sums of a random multiplicative function}

\author{Besfort Shala}
\address{Bâtiment Sophie Germain (Université Paris Cité), 8 place Aurélie Nemours, 75013 Paris, France}
\curraddr{}
\email{besfort.shala@imj-prg.fr}
\thanks{}


\date{}

\dedicatory{}

\begin{abstract} Let $f$ be a Rademacher or Steinhaus random multiplicative function. For various arithmetically interesting subsets $\mathcal A\subseteq [1, N]\cap\mathbb N$ such that the distribution of $\sum_{n\in \mathcal A} f(n)$ is approximately Gaussian, we develop a general framework to understand the large fluctuations of the sum. This extends the general central limit theorem framework of Soundararajan and Xu.

In the case when $\mathcal A = (N-H, N]$ is a short interval with admissible $H=H(N)$, we show that almost surely
\begin{equation*}
    \limsup_{N\to\infty} \frac{\big\lvert\sum_{N-H<n\leq N} f(n)\big\rvert}{\sqrt{H\log \frac{N}H{}}}>0.
\end{equation*} When $\mathcal A$ is the set of values of an admissible polynomial $P\in\mathbb Z[x]$, we extend work of Klurman, Shkredov, and Xu, as well as Chinis and the author, showing that almost surely 
\begin{equation*}
   \limsup_{N\to\infty} \frac{\big\lvert\sum_{n\leq N} f(P(n))\big\rvert}{\sqrt{N \log\log N}}>0,
\end{equation*} even when $P$ is a product of linear factors over $\mathbb Q$. In this case, we also establish the corresponding almost sure upper bound, matching the law of iterated logarithm. 

An important ingredient in our work is bounding the Kantorovich--Wasserstein distance by means of a quantitative martingale central limit theorem. 
\end{abstract}

\maketitle

\section{Introduction} 

Since their introduction by Wintner \cite{Wintner1944}, random multiplicative functions have attracted a lot of attention in number theory. 
\begin{definition}
    A \emph{Steinhaus} random multiplicative function $f$ is a sequence of random variables $f(1), f(2), \ldots$ such that $f(mn)=f(m)f(n)$ for all positive integers $m$ and $n$, and $(f(p))_p$ over primes $p$ is a sequence of independent random variables, each uniformly distributed on the unit circle.

    A \emph{Rademacher} random multiplicative function $f$ is a sequence of random variables $f(1), f(2),$ $\ldots$  supported on square-free integers such that $f(mn) = f(m)f(n)$ whenever $m$ and $n$ are coprime, and $(f(p))_p$ over primes $p$ is a sequence of independent random variables, each uniformly distributed on $\{-1, 1\}$. 
\end{definition}
Although random multiplicative functions are useful in modeling deterministic objects such as Dirichlet and Archimedean characters, they have become increasingly interesting to study in their own right. Building on earlier work of Hal\'asz \cite{Halasz1977} and Hal\'asz and R\'enyi \cite{HalaszRenyi1961}, partial sums of a random multiplicative function $\sum_{n\leq N} f(n)$ have been extensively studied by Harper. Being a sum of random variables, it is natural to ask whether $\frac{1}{\sqrt N}\sum_{n\leq N} f(n)$ has an approximately Gaussian distribution as $N\to\infty$. 
Harper \cite{Harper2013} showed that this is not the case and in fact, he later showed \cite{Harper2020_MomentsRandomMultiplicativeI} that the typical size of these sums is of order $\sqrt N/(\log\log N)^{\frac14}$. One might therefore hope to obtain 
an interesting limiting distribution after normalizing by this factor instead. This is a delicate matter closely related to Gaussian multiplicative chaos, and has been resolved recently in the work of Gorodetsky and Wong \cite{gorodetsky_wong_2025_limiting} in the Steinhaus case. 

From an arithmetic point of view, it is interesting to understand sums of a random multiplicative function when the summands are restricted to subsets of positive integers. This will be our focus in this paper. Examples of such restrictions that have been considered previously include: 
\begin{itemize}\item[(i)] integers with an atypically small number of prime factors (Hough \cite{Hough2011}, Harper \cite{Harper2013}), 
\item[(ii)] integers in a short interval (Chatterjee and Soundararajan \cite{ChatterjeeSoundararajan2012}, Soundararajan and Xu \cite{soundararajan2023clt}, Pandey, Wang, and Xu \cite{PandeyWangXu2024}), 
\item[(iii)] sums of two squares in a short interval and shifted primes $p-k$ for $k\neq 0$ (Soundararajan and Xu \cite{soundararajan2023clt}), 
\item[(iv)] polynomial images (Najnudel \cite{najnudel2020consecutive}, Klurman, Shkredov, and Xu \cite{KlurmanShkredovXu2023}, Chinis and the author \cite{ChinisShala2025}).
\end{itemize} 
In each of the aforementioned works, the distribution of the sums turns out to be approximately Gaussian, demonstrating that multiplicativity seems to interfere less with independence when the summands are restricted to a less ``multiplicatively structured'' set. In this paper, we extend the general framework of Soundararajan and Xu \cite{soundararajan2023clt} on central limit theorems for random multiplicative functions to the setting of large fluctuations.

There has been considerable work on understanding almost sure large fluctuations for the full sum $\sum_{n\leq N} f(n)$ and related variants. Harper \cite{Harper2023_ASLF} showed that for any (slowly growing) function $V(N)\to\infty$, almost surely there is a sequence of positive integers $N_k\to\infty$ such that 
\begin{equation*}
    \left\lvert\sum_{n\leq N_k}f(n)\right\rvert\gg\sqrt{N_k}\frac{(\log\log N_k)^{\frac14}}{V(N_k)}.
\end{equation*}
This matches what one expects from the law of iterated logarithm, up to the factor $V(N)$. Proving the corresponding upper bound remained a challenge for longer. Partial results with an upper bound of the right shape (namely $\sqrt N$ times a power of $\log\log N$) have been proved by Lau, Tenenbaum and Wu \cite{LauTenenbaumWu2013MeanValuesRMF}, Caich \cite{Caich2023}, as well as recently by Durkan and Pearce-Crump \cite{DurkanPearceCrump2026}, and Verreault \cite{Verreault2026} with the corresponding essentially sharp almost sure upper bound.

In the case when the summands are demanded to have a prime factor bigger than $\sqrt N$, Mastrostefano \cite{Mastrostefano2022} proved an essentially sharp almost sure upper bound with the exponent $\frac{1}{4}+\epsilon$, with the corresponding almost sure lower bound following by work of Harper \cite{Harper2023_ASLF}; see also the work of Hardy \cite{Hardy2025LargePrime} on the distributional aspect in this case. A weighted variant $\sum_{n\leq N} \frac{f(n)}{\sqrt n}$
has been considered by Aymone, Heap and Zhao \cite{aymone2021partial}, and Hardy \cite{Hardy2023AlmostSure} in the Steinhaus case, and Atherfold \cite{Atherfold2026} in the Rademacher case.

Recently, Hoban, Shah, Ismail, Verreault, and Zaman \cite{HobanShahIsmailVerreaultZaman2025} have extended the framework of Soundararajan and Xu \cite{soundararajan2023clt} in a different direction, namely in the function field setting. It is likely that the methods of this paper, combined with the work in \cite{HobanShahIsmailVerreaultZaman2025}, can be used to prove analogous results over function fields.

We are now ready to state the main results of this paper. Our focus will be on two main examples: (i) polynomial images, and (ii) integers in a short interval. However, our method is rather general, so it should be able to flexibly handle several examples; see Section \ref{section:generalSetup} for the general but rather technical results. 

\subsection{Polynomial images} This case is particularly interesting due to its connection with conjectures of Chowla and Elliott \cite{Chowla1965, Elliott1992, Elliott1994} on correlations of multiplicative functions. See the introductions of \cite{KlurmanShkredovXu2023, ChinisShala2025} for a more detailed discussion in the setting of random multiplicative functions, where central limit theorems were proved.

\begin{theoremL}\label{thm:asboundsPoly}
    Let $f$ be a  Rademacher or Steinhaus random multiplicative function. Let $P\in\mathbb Z[x]$ be a polynomial taking only positive values.\footnote{This is not a restriction, for if $P(x)$ has a positive leading coefficient, its values are eventually all positive and we may replace $P(x)$ by $P(x+a)$ for some positive integer $a$ depending on $P$.} In the Rademacher case, assume $P\in\mathbb Z[x]$ takes a positive proportion of square-free values.\footnote{The result of Booker and Browning \cite{BookerBrowning2016} supplies this unconditionally for any square-free polynomial with no local obstructions, provided the degree of each of its irreducible factors is at most $3$. The $abc$ conjecture removes the degree restriction, by work of Granville \cite{Granville1998}.} In the Steinhaus case, assume $P(x)$ is not of the form $w(x+c)^d$ for $w\in\mathbb Z, c\in\mathbb Q$. Then almost surely we have $$\left\lvert\sum_{n\leq N} f(P(n))\right\rvert \ll \sqrt{N\log\log N}.$$ Moreover, almost surely there exists a sequence $N_k\to\infty$ such that 
    \begin{equation*}
        \left\lvert\sum_{n\leq N_k} f(P(n))\right\rvert\gg \sqrt{N_k\log\log N_k}.
    \end{equation*}
\end{theoremL}

The almost sure upper bound in Theorem \ref{thm:asboundsPoly} is new, whereas the lower bound is an extension of results of Klurman, Shkredov, and Xu \cite{KlurmanShkredovXu2023} in the Steinhaus case, and Chinis and the author \cite{ChinisShala2025} in the Rademacher case, to polynomials that split as a product of linear factors over $\mathbb Q$. Indeed, the method used in \cite{KlurmanShkredovXu2023, ChinisShala2025} relied on the fact that for a polynomial $P$ that has an irreducible factor of degree at least two, there are many integers $n$ for which $P(n)$ has a prime factor greater than $n\log n$. In the Rademacher case, the class of polynomials for which we prove Theorem \ref{thm:asboundsPoly} is wider than that considered\footnote{Nonetheless, thanks to an observation due to Ihor Pylaiev, the proof in \cite{ChinisShala2025} may be straightforwardly adapted to extend the class of admissible polynomials.} in \cite{ChinisShala2025}, in light of the recent results of la Bretèche, Wang, and Xu \cite{deLaBretecheWangXu2026}, as well as Pylaiev \cite{Pylaiev2026}.

To demonstrate how this was used, let us consider $P(n) = n^2+1$. By the aforementioned fact, one can ``pull out'' a unique large prime factor for such $n$ and still capture a good ``bulk'' of the sum for infinitely many (sparse) $N$, that is 
\begin{equation*}
    \sum_{n\leq N}f(n^2+1)\approx \sum_{N\log N < p\leq N^2}f(p)\sum_{\substack{n\leq N \\ p\mid n^2+1}}f\left(\frac{n^2+1}{p}\right).
\end{equation*}

Upon conditioning on the small primes, this is a weighted sum of independent random variables, which is then handled using intricate combinatorial arguments and Gaussian approximation. 

Indeed, the method described above cannot be directly applied to the case when $P$ is a product of linear factors, for all its prime factors are of size at most linear in $N$.
In this paper, we will use a different method that does not involve any combinatorial splitting of the sum or conditioning. Rather, we will fully utilize the martingale structure of the entire sum -- we will elaborate on this in the next section. In particular, provided enough number-theoretic input (see the work of Martin \cite{Martin2002}), our method would also work to prove almost sure large fluctuations of $N$-smooth\footnote{An integer is called $y$-smooth if all of the prime factors dividing it are at most $y$.} values of irreducible polynomials.

\subsection{Integers in a short interval} A central limit theorem for sums of a random multiplicative function over integers in a short interval $(N-H, N]$ was first proved by Chatterjee and Soundararajan \cite{ChatterjeeSoundararajan2012} when $H=o\left(\frac{N}{\log N}\right)$, using essentially only second and fourth moment information. This was later extended to $H\leq \frac{N}{(\log N)^{c}}$ with $c>2\log2-1$ by Soundararajan and Xu \cite{soundararajan2023clt}, entering a regime of $H$ where naively the fourth moment of the sum blows up. However, this is bypassed by throwing out $0\%$ of the integers in the short interval with atypically many prime factors, resulting in a controlled fourth moment. 

We prove an almost sure lower bound in the same regime. 

\begin{theoremL}\label{thm:asBoundsShort}
    Let $f$ be a Rademacher or Steinhaus random multiplicative function and suppose that $H=H(N)$ is a smooth, increasing, and concave function, such that $N^{\frac{11}{15}}\leq H(N)\leq N/(\log N)^c$
    for some constant $c>2\log 2-1$. 
    Then almost surely there exists a sequence $N_k\to\infty$ with corresponding $H_k = H(N_k)$ such that 
    \begin{equation*}
        \left\lvert\sum_{N_k-H_k < n \leq N_k}f(n)\right\rvert\gg \sqrt{H_k\log\frac{N_k}{H_k}}.
    \end{equation*}
\end{theoremL}

\begin{remark}
    The fluctuations of order $\sqrt{\log\frac{N}{H}}$ may seem surprisingly large at first, especially in a naive comparison with the law of iterated logarithm. However, this is explained by the fact that short sums of length $H$ decorrelate much faster, and indeed roughly $N/H$ essentially independent sums fit in a ``window'' of size $N$. In agreement with the maximum of independent Gaussian random variables, the maximum of these sums has size roughly $\sqrt{H\log\frac{N}{H}}$. 
\end{remark}

\begin{remark}
    The conditions on $H$ are imposed simply for convenience -- any standard choices such as $H=N^{\alpha}$ for $\frac{11}{15}\leq\alpha < 1$, $H = N/\exp\left(\sqrt{\log N}\right)$, or $H = N/(\log N)^A$ for some constant $A>2\log 2-1$, satisfy these conditions (for sufficiently large $N$). 
    
    We have not attempted to optimize the exponent $\frac{11}{15}$ -- it is likely possible to bring it down to $\frac{3}{5}$. However, the lower bound on $H$ should not be essential: for very small $H\ll \log N$ the large fluctuations are of size $H$ and captured by almost sure long runs of suitable signs over the primes (see also the work of Erdős and Rényi \cite{ErdosRenyi1970}), whereas if $\log N\ll H\ll N^{\frac{11}{15}}$, our method should work (with significant simplifications when $H\ll N^{\frac{1}{2}}$), but the restriction stems from a result on square-free smooth integers in short intervals; see Lemma \ref{lemma:unSmoothSquarefree}. Nonetheless, since we are only proving a lower bound and therefore can pick convenient scales to work with, it is likely possible to circumvent this using a suitable adaptation of the work of Gorodetsky, Mangerel, and Rodgers \cite{GorodetskyMangerelRodgers2023}. However, we have opted for simplicity, especially since our result captures the most interesting regime where there is a transition of the large fluctuations from $\sqrt{\log N}$ down to $\sqrt{\log \log N}$.\footnote{We thank Sarvagya Jain for discussions pertaining to this remark, as well as Sacha Mangerel for pointing out that the Matom\"aki--Radziwiłł machinery (as referenced in a previous version of this remark) is unnecessary.}
\end{remark}

It would be interesting to prove the corresponding almost sure upper bound in any regime of $H$. For rather large $H$, it is likely that the sophisticated ideas in \cite{LauTenenbaumWu2013MeanValuesRMF, Mastrostefano2022, Hardy2023AlmostSure, Atherfold2025} relating the short sum $\sum_{N-H<n\leq N} f(n)$ to a random Euler product could be useful. It is unclear, however, whether this would result in a sharp (or almost sharp) upper bound in any regime of $H$ -- this is an ongoing investigation. Recent work of Harper, Soundararajan, and Xu \cite{HarperSoundararajanXu2026}
establishes a central limit theorem in the range $H = o(N)$, although with a normalization of the sum that differs from just the standard deviation. It would be interesting to determine whether our method would succeed in proving almost sure lower bounds of the right order of magnitude in this extended range. See also the work of Caich \cite{Caich2024randomshort} for more on the intricate behavior of short sums in this regime.

\section*{Acknowledgements}
I would like to express my deep gratitude to my supervisors, Jonathan Bober and Oleksiy Klurman, for their continued support, encouragement, delightful discussions, and their careful reading and feedback on early drafts of this work. I am also grateful to Sacha Mangerel for a thorough examination of my thesis, of which this paper was an integral part. Moreover, I would like to thank Joseph Najnudel for a fruitful discussion regarding Gaussian approximation. Last but not least, I am thankful to Christopher Atherfold, Seth Hardy, Sun Kai Leung, and Neo Tardy, as well as anonymous referees, for enlightening technical conversations and/or their comments on earlier drafts of the paper. The author is funded by a University of Bristol PhD scholarship.

\section*{AI Declaration}
The first three arXiv versions of this paper have been written entirely independently of any AI tools. For the current version, the sole contribution of ChatGPT 6.2 Astra on the level of ideas was observing that only two-dimensional quantitative central limit theorems suffice for the bootstrapping in the almost sure lower bound results, leading to Corollary \ref{cor:pairwiseNormalComparison}; this is heavily in the spirit of the argument that the author was already using in the proof for short intervals in previous versions.  

Moreover, ChatGPT-6 Astra and ChatGPT-6.2 Astra were used to identify errors throughout the paper, all of which were checked and corrected by the author; in some instances ChatGPT-6.2 Astra was used to execute the write-up of the routine correction. 

\section{Overview of the proofs}
\subsection{Martingale structure} Our starting point is the observation of Harper \cite{Harper2013} that sums of a random multiplicative function have the structure of a martingale with respect to the filtration induced by the largest prime factor, namely 
\begin{equation*}
    \sum_{n\in\mathcal A}f(n) = \sum_{p^{\alpha}\leq N}f(p^{\alpha})\sum_{\substack{n\in\mathcal A \\ P^+(n)=p \\ p^{\alpha}\parallel n}}f\left(\frac{n}{p^{\alpha}}\right).
\end{equation*}
This observation was utilized in almost all subsequent works on random multiplicative functions. For us, its usefulness lies in the fact that it morally reduces the understanding of $\sum_{n\in\mathcal A}f(n)$ to evaluating the second and fourth moments of the sum, provided they do not blow up. More precisely, we will apply the following quantitative result from the book of Hall and Heyde \cite{HallHeyde1980}, yielding a bound on the Kolmogorov distance between our martingale and a Gaussian. 
\begin{prop}[\cite{HallHeyde1980}, Theorem 3.9]\label{prop:quantCLT}
    Let $X_1, X_2, \ldots, X_n$ be a real mean-zero martingale difference sequence such that 
    $$\sum_{i=1}^n \mathbb E\lvert X_i\rvert^4 + \mathbb E\left\lvert\sum_{i=1}^n X_i^2 - 1\right\rvert^2\leq \frac{1}{3}.$$
    Let $S_n = \sum_{i=1}^n X_i$. Then 
    \begin{equation*}
\left\lvert\mathbb P(S_n\leq x) - \frac{1}{\sqrt{2\pi}}\int_{-\infty}^x e^{-t^2/2}\emph{d}t\right\rvert\ll \frac{\left(\sum_{i=1}^n \mathbb E\lvert X_i\rvert^4 + \mathbb E\left\lvert\sum_{i=1}^n X_i^2 - 1\right\rvert^2\right)^{1/5}}{1+\lvert x\rvert^{\frac{16}{5}}}.
    \end{equation*}
\end{prop}

\begin{proof}
    This follows directly from \cite[Theorem 3.9]{HallHeyde1980}, upon noticing that
    $$L_n :=\sum_{i=1}^n \mathbb E\lvert X_i\rvert^4 + \mathbb E\left\lvert \sum_{i=1}^n \mathbb E(X_i^2\mid \mathcal F_{i-1})-1\right\rvert^2\leq 3\left(\sum_{i=1}^n \mathbb E\lvert X_i\rvert^4 + \mathbb E\left\lvert\sum_{i=1}^n X_i^2 - 1\right\rvert^2 \right),$$ where the $\mathcal F_i$ denote the corresponding filtration of the martingale. This follows from 
    $$\mathbb E\left\lvert \sum_{i=1}^n (X_i^2 - \mathbb E(X_i^2\mid \mathcal F_{i-1}))\right\rvert^2 = \sum_{i=1}^n \mathbb E(X_i^4 - (\mathbb E(X_i^2\mid \mathcal F_{i-1}))^2)\leq \sum_{i=1}^n \mathbb E\lvert X_i\rvert^4.$$
    As such, the hypothesis of the theorem that $L_n\leq 1$ is clearly satisfied by the condition that the right-hand side bounding $L_n$ is at most $\frac{1}{3}$.
\end{proof}

\subsection{Almost sure upper bound} Equipped with the above result, we may readily prove an almost sure upper bound for $\sum_{n\in\mathcal A} f(n)$, provided our approximation is suitably strong. In order to access the tails at $x\approx C\sqrt{\log k}$ which occur with probability $\approx k^{-C^2/2}$ for the Gaussian, we will need an approximation that is at least this good. We are able to achieve this in the case of polynomial images, but not short intervals -- we will discuss this in more detail in the following subsections. Using our quantitative approximation at several spread out scales together with a slow variation property between the scales (see Proposition \ref{prop:slowVar} and Lemma \ref{lemma:slowVarPoly}), the Borel-Cantelli lemma establishes the almost sure upper bound. 

\subsection{Almost sure lower bound} Establishing an almost sure lower bound is where the main novelty of this paper lies. The goal here is obtaining a quantitative multivariate central limit theorem for 
\begin{equation*}
    \left(\frac{1}{\sqrt{\lvert\mathcal A_1\rvert}}\sum_{n\in\mathcal A_1}f(n), \frac{1}{\sqrt{\lvert\mathcal A_2\rvert}}\sum_{n\in\mathcal A_2}f(n), \ldots, \frac{1}{\sqrt{\lvert\mathcal A_k\rvert}}\sum_{n\in\mathcal A_k}f(n) \right), 
\end{equation*}
where each $\mathcal A_l\subseteq [1, N_l]\cap \mathbb N$ for various integers $N_1, N_2, \ldots, N_k$. To do this, roughly speaking, we will first apply Proposition \ref{prop:quantCLT} to linear combinations of the normalized sums $S_l := \lvert\mathcal A_l\rvert^{-\frac{1}{2}}\sum_{n\in\mathcal A_l} f(n)$, establishing approximate joint Gaussianity. We will then bootstrap this to a quantitative multivariate central limit theorem, by using the following recent key result of Bobkov and G\"otze \cite{BobkovGoetze2024} to bound the Kantorovich--Wasserstein distance between the multivariate distributions in terms of the Kolmogorov distances of the corresponding linear combinations. 

\begin{prop}[\cite{BobkovGoetze2024}, Theorem 1.1]\label{prop:multivarCLT} Let $X$ and $Y$ be $\mathbb R^k$-valued random variables with $\|X\|_p\leq b$ and $\|Y\|_p\leq b$ for some $p>1$ and $b\geq 0$, where $\|Z\|_p = \left( \mathbb E\lvert Z\rvert^p\right)^{\frac{1}{p}}$. Then for $q$ such that $1/p + 1/q=1$, we have
\begin{multline*}
    \sup_{\|u\|_{\emph{Lip}}\leq 1} \lvert \mathbb EuX - \mathbb EuY\rvert \\ \ll b^{1-2/(qk+2)} \left(\sup_{\lvert \theta\rvert = 1} \int_{-\infty}^\infty \left\lvert\mathbb P\left(\sum_{l=1}^k \theta_l X_l\leq x \right)- \mathbb P\left(\sum_{l=1}^k \theta_l Y_l\leq x\right)\right\rvert\text{d} x\right)^{2/(kq+2)}. 
\end{multline*} Here the supremum is taken over all functions $u:\mathbb R^k\to\mathbb R$ with Lipschitz semi-norm at most $1$, and $\lvert\theta\rvert$ denotes the Euclidean norm of $\theta = (\theta_1, \theta_2, \ldots, \theta_k)\in\mathbb R^k$. 
\end{prop}

In fact, it turns out that for the almost sure lower bound results, it suffices to utilize the above strategy only for two-dimensional vectors. This is due to the following simple observation. If $s:\mathbb R\to [0, 1]$ is a continuous function that ``detects'' whether its argument is at most $t$, note that if $S_l\leq t$ for all $1\leq l\leq k$, then $\frac{1}{k}\sum_{l=1}^k s(S_l)=0$. Applying Markov's inequality to $\frac{1}{k}\sum_{l=1}^k s(S_l) -s(Y)$ with the second moment, where $Y$ is a standard Gaussian random variable, we obtain that the probability of the former event is bounded above by 
$$\frac{1}{(\mathbb Es(Y))^2}\mathbb E\left(\frac{1}{k}\sum_{l=1}^k s(S_{l})-\mathbb Es(Y) \right)^2.$$ Upon expanding the square, each of the terms may be bounded by the bootstrapped quantitative two-variable central limit theorem. Nonetheless, the general $k$-dimensional bootstrapping establishes a strictly stronger result which may be interesting in its own right, albeit with some quantitative loss in $k$.

This strategy is a more direct approach compared to the use of Stein's method via exchangeable pairs in the work of Harper \cite{Harper2013ZetaMax, Harper2023_ASLF}. In particular, the author was unable to adapt the latter method to handle the complex dependencies between large primes $p$ and $q$ such that $p\mid n$ and $q\mid n+1$ when dealing with $\sum_{n\leq N} f(n(n+1))$, say. The advantage of our approach is that we obtain a quantitative multivariate central limit theorem purely from the one-dimensional martingale structure. See also the recent work of Kowalski and Untrau \cite{KowalskiUntrau2025Wasserstein}, Humphries \cite{Humphries2025QuantitativeEquidistribution}, as well as earlier work of Saksman and Webb \cite{saksman2020riemann} on the use of the Kantorovich--Wasserstein distance in number theory. 

Equipped with a quantitative multivariate central limit theorem, we may compare the distribution of our $k$ sums to the distribution of $k$ Gaussian random variables with a prescribed covariance structure governed by the arithmetic of the set $\mathcal A$. Given this information, we may readily establish the almost sure lower bound by an application of Proposition \ref{prop:multivarCLT} and the Borel-Cantelli lemma. 

\subsection{Arithmetic input} After applying Proposition \ref{prop:quantCLT} to $\sum_l\theta_l\lvert \mathcal A_{l}\rvert^{-\frac{1}{2}}\sum_{n\in\mathcal A_l}f(n)$ for Rade-macher (or Steinhaus) $f$, what we essentially have to control is the number of solutions to the equation $$n_1n_2n_3n_4=\square \text{ (or } n_1n_2=n_3n_4\text{)}$$ with $n_i\in\mathcal A_{l_i}$ for $1\leq i\leq 4$ coming from (at most) four different scales. Henceforth we will refer to this as the \emph{fourth moment equation}. 
In each of the examples we consider in the paper, we would like for this equation to ``only'' (in an asymptotic sense) have the trivial solutions where the $n_i$ are equal in pairs, contributing approximately $\lvert \mathcal A_{l_i}\rvert\lvert\mathcal A_{l_j}\rvert$ solutions for $1\leq i, j\leq 4$. This means that we would like to rule out the many non-trivial solutions that occur in the full ranges of integers. In the following subsections, we will describe how we achieve this, as well as the subtleties in the previously described method for the examples we consider in this paper. 

\subsubsection{Polynomial images} This case will be our most streamlined application of the method as described above. 
We will utilize the main results of \cite{KlurmanShkredovXu2023, ChinisShala2025}, giving a power-saving bound on the number of non-trivial solutions to the fourth moment equations  $$P(n_1)P(n_2)P(n_3)P(n_4) = \square \text{ (or  } P(n_1)P(n_2) = P(n_3)P(n_4)\text{)}$$ with all $n_i\leq N$. This allows us to choose the $N_l$ for $1\leq l\leq k$ so that we can control the covariances of the sums $N_l^{-\frac{1}{2}}\sum_{n\leq N_l} f(P(n))$ quite crudely, by simply bounding the number of non-trivial solutions to the equation with $n_i$ ranging up to the maximum of the $N_{l_i}$. 

However, the martingale structure that we use forces us to combine the power-saving bound above with upper bounds for (very) smooth values of polynomial images, resulting in a quantitative bound of the form $\exp\left(-\sqrt{\log X}\right)$ in the martingale central limit theorem (where $X$ is so that $\log N_l\asymp \log X$ for all $1\leq l\leq k$); see Corollary \ref{cor:quantCLTpoly}. Nonetheless, this is certainly strong enough to allow us to access the tails that occur with probability $\approx (\log X)^{-C^2/2}$, leading to the almost sure upper bound as described before.

\begin{remark}\label{rmk: wangXumoments} One could also obtain the almost sure upper bound in the Steinhaus case by evaluating the $2k$-th  moment of $\sum_{n\leq N} f(P(n))$ for $k\approx \log\log N$, as was done in the work of Wang and Xu \cite{Wang-Xu}. A careful inspection of their proof shows that this is indeed the uniformity they obtain. This appears harder to do in the Rademacher case.\footnote{We thank Christopher Atherfold for discussions pertaining to this remark.}  \end{remark}

For the lower bound, passing to the multivariate central limit theorem with $k$-dimensional vectors would result in a further loss, namely a bound of the form $\sqrt k\exp\left(-\frac{\sqrt{\log X}}{k}\right)$ for the approximation. In this case, it is sufficient to choose $k=(\log X)^{\varepsilon_0}$ for some small $\varepsilon_0>0$, so this would result in large fluctuations of order $\approx \sqrt{2\log k}\approx \sqrt{\log\log X}$ almost surely. However, as explained above, we instead utilize the observation of applying the bootstrapped two-dimensional martingale central limit theorems to roughly $k^2$ pairs.

\subsubsection{Integers in a short interval} Attempting to follow the strategy as in the polynomial case here runs into the following difficulties. Firstly, the bounds we have for the number of non-trivial solutions to $n_1n_2n_3n_4=\square$ (or $n_1n_2=n_3n_4$ in the Steinhaus case) are much weaker and get worse as $H$ gets bigger. Secondly, even when $H=N^{\alpha}$ in which case we have power-saving bounds on the number of non-trivial solutions, as before, the very smooth solutions (including the trivial ones) yield a martingale central limit theorem with a quantitative bound of shape roughly $\exp\left(-\sqrt{\log X}\right)$. This is not enough to capture the large fluctuations of size $C\sqrt{\log X}$ that occur with probability roughly $X^{-C^2/2}.$ 

\begin{remark}
    When $k=1$, the above strategy nonetheless yields a significant quantitative improvement of the central limit theorem of Chatterjee and Soundararajan \cite{ChatterjeeSoundararajan2012}, but we do not state this in the paper. We note that the result of Soundararajan and Xu \cite{soundararajan2023clt} is also quantitative, but stated in a different form in terms of the characteristic function, with an apparent large loss of $e^{t^2/2}$. However, it appears that by being less crude with the bound $|1+it|\leq e^{t^2/2}$ for large $t$ in their work, applying Esseen's inequality (truncated Fourier inversion) to their result would probably recover a quantitative central limit theorem of similar strength. 
\end{remark}

To bypass these issues, we condition on the values of $f$ on the small primes, and approximate the sums $\sum_{N_l-H_l<n\leq N_l} f(n)$ for most $1\leq l\leq k$ with a weighted sum of independent random variables $f(p)$ for large primes $p$. This will be advantageous, as the weaker saving we have on the number of non-trivial solutions to the fourth moment equation will only be used to control the $\approx k^2$ conditional covariances, and will not play a role in the quantitative strength of the conditional martingale central limit theorem.

In controlling the conditional covariances by counting solutions to the equation $n_1n_2n_3n_4 = \square$ (or $n_1n_2=n_3n_4$) where the $n_i$ are potentially coming from different scales, we have to be mindful about the choice of scales. This is due to the many non-trivial solutions (of the same order of magnitude as the trivial solutions) which are obtained by scaling, such as solutions of the form $(n_1)(2n_1)(n_3)(2n_3)=\square$, say. 

When $H$ is greater than roughly $\frac{N}{\log N}$, we implement the strategy used in the work of Soundararajan and Xu \cite{soundararajan2023clt}, throwing out integers with atypically many prime factors, resulting in the desired control of the fourth moment equation as long as $H\leq \frac{N}{(\log N)^c}$ with $c>2\log 2-1$; see Proposition \ref{prop:CountAt4Scales}. For us, however, it is crucial that we obtain a sufficiently strong quantitative bound for the number of such integers; see Lemmas \ref{lemma:tooManyPrimes} and  \ref{lemma:ignoreTooManyPrimesShort}.

\subsubsection{Other examples and limitations} The method used in this paper may be used to prove almost sure lower bounds on large fluctuations of other examples, such as: shifted primes, sums of squares in short intervals \cite{soundararajan2023clt}, rough integers \cite{Xu2024btran}, and exponential sums with random multiplicative coefficients $\sum_{n\leq N} f(n)e(n\theta)$ \cite{BenatarNishryRodgers2022, Hardy2024imrn, soundararajan2023clt}. However, we are not able to handle the case of restricted number of prime factors $\omega(n)\leq w=o(\log\log N)$ for $n\le N$ considered in \cite{Hough2011, Harper2013}, for the number of non-trivial solutions to the fourth moment equation there is very large. Even when $w>1$ is finite, the saving over the number of trivial solutions is at most double-logarithmic in $N$, resulting in a very small choice for the number of scales $k$ that we can hope to use. It would be interesting to determine whether this is a reflection of the truth, that is, whether the large fluctuations of $\sum_{n\leq N,\text{ } \omega(n)\leq w} f(n)$ differ from what would be expected from the approximate Gaussian distribution. 

In another direction, it would also be interesting to prove almost sure upper bounds for other examples, such as shifted primes. 
In contrast with short intervals, say, it is not clear how one would proceed here, as this sum likely cannot be related to a random Euler product --- this is an ongoing investigation. 

\section{Auxiliary Results}

In this section, we collect and/or prove some auxiliary results that we will use throughout our proofs.

\subsection{Number theory results} We start by stating a result of Nair and Tenenbaum \cite{NairTenenbaum1998} (see also the work of Henriot \cite{Henriot2012_NairTenenbaumUniform}), in a less general form that will be sufficient for us. This is a generalization of a well-known lemma of Shiu. 

\begin{prop}\label{prop:nair-tenenbaum}
    Let $F$ be a non-negative multiplicative function such that there exist constants $A, B\geq 1$ and $\epsilon >0$ such that $F(m)\leq \min\{A^{\Omega(m)}, Bm^{\epsilon/3}\}$ for all $m\in\mathbb N$. Let $Q\in\mathbb Z[x]$ be a fixed polynomial of degree $d$ with no fixed prime factor, and let $\rho(n)$ be the number of solutions to $Q(x)\equiv 0\pmod n$. Suppose that $Q$ has $r$ distinct irreducible factors $Q_1, Q_2, \ldots, Q_r$, and write $Q = \prod_{i=1}^r Q_i^{\gamma_i}$. Let $\rho_i(n)$ be the number of solutions to $Q_i(x)\equiv 0\pmod n$ for $1\leq i\leq r$. Denote the discriminant of the square-free kernel of $Q$ by $D$. If $\epsilon<\frac{1}{8d^2}$, then 
    \begin{multline*}
        \sum_{N<n\leq N+H} F(\lvert Q(n)\rvert ) \\ \ll H\prod_{p\leq N}\left(1-\frac{\rho(p)}{p}\right) \sum_{\substack{n_1^{\gamma_1}n_2^{\gamma_2}\cdots n_r^{\gamma_r}\leq N \\ (n_i, n_j)=1 \text{ }\forall i\neq j \\ (n_i, D)=1 \text{ }\forall i}} F(n_1^{\gamma_1}\cdots n_r^{\gamma_r})\frac{\rho_1(n_1)\rho_2(n_2)\cdots \rho_r(n_r)}{n_1n_2\cdots n_r},
    \end{multline*}
    whenever $N$ is large enough, and $N^{4d^2\epsilon}\leq H\leq N$. 
\end{prop}

We now record a few corollaries of this result. 

\begin{lemma}\label{lemma:PolyBound}
    Let $Q\in\mathbb Z[x]$ be a non-zero polynomial. There exists a constant\footnote{In fact Proposition \ref{prop:nair-tenenbaum} gives a much more precise description of $c_Q$, but we will not need it here.} $c_Q>0$ such that 
    \begin{equation*}
        \sum_{N<n\leq N+H} \tau_3(\lvert Q(n)\rvert)\ll H (\log N)^{c_Q}
    \end{equation*}
    whenever $N^{1/3}\leq H \leq N$, where $\tau_3$ is the $3$-fold divisor function.
\end{lemma}

\begin{proof}
The claim is immediate if $Q$ is constant. Suppose that $g=\deg Q\geq 1$,
let $a$ be the leading coefficient of $Q$, and put $b=g!\lvert a\rvert$.
The polynomial $P(x)=b^g Q(x/b)/a$ is monic with integer coefficients
and has no fixed prime factor. Indeed, $P(x)\equiv x^g\pmod p$ for every
prime $p\leq g$, while for $p>g$ the polynomial $P$ has at most $g$
roots modulo $p$. Since $Q(n)$ divides $P(bn)$, we have, for all
sufficiently large $N$,
\[
    \sum_{N<n\leq N+H}\tau_3(\lvert Q(n)\rvert)
    \leq \sum_{bN<m\leq bN+bH}\tau_3(\lvert P(m)\rvert).
\]
As $(bN)^{1/3}\leq bH\leq bN$ and $b$ depends only on $Q$, we may
therefore assume that $Q$ is monic and has no fixed prime factor.

Write $Q=\prod_{i=1}^r Q_i^{\gamma_i}$, where the $Q_i$ are distinct
monic irreducible polynomials, and put $d_i=\deg Q_i$. Let $D$ be the
discriminant of $\prod_{i=1}^r Q_i$, and let $\rho_i(m)$ denote the
number of roots of $Q_i$ modulo $m$. Take $\epsilon=1/(12g^2)$, so
that $N^{4g^2\epsilon}=N^{1/3}\leq H$. We have
$\tau_3(n)\leq 3^{\Omega(n)}$ and
$\tau_3(n)\ll_{\epsilon} n^{\epsilon/3}$. Applying Proposition
\ref{prop:nair-tenenbaum}, bounding the sieve product by $1$, and
using the submultiplicativity of $\tau_3$, we obtain
\begin{align*}
    \sum_{N<n\leq N+H}\tau_3(\lvert Q(n)\rvert)
    &\ll_Q H\prod_{i=1}^r\prod_{\substack{p\leq N\\p\nmid D}}
    \left(1+\sum_{\nu\geq 1}
    \frac{\tau_3(p^{\gamma_i\nu})\rho_i(p^\nu)}{p^\nu}\right)\\
    &\ll_Q H\exp\left(\sum_{i=1}^r d_i\binom{\gamma_i+2}{2}
    \sum_{p\leq N}\frac{1}{p}\right)
    \ll_Q H(\log N)^{c_Q},
\end{align*}
where $c_Q=\sum_{i=1}^r d_i\binom{\gamma_i+2}{2}$. Indeed, for
$p\nmid D$ all roots of $Q_i$ modulo $p$ are simple, so that
$\rho_i(p^\nu)=\rho_i(p)\leq d_i$ for every $\nu\geq 1$. Since
$\tau_3(p^{\gamma_i\nu})=\binom{\gamma_i\nu+2}{2}$, the terms with
$\nu\geq 2$ in each Euler factor contribute $O_Q(p^{-2})$.
\end{proof}

\begin{lemma}\label{lemma:tooManyPrimes}
    Let $0<\varepsilon<1/5$, and suppose that $H=H(N)$ satisfies
    $H\geq N^{5\varepsilon}$ and $H=o(N)$ as $N\to\infty$.
    Let $\Omega(n)$ denote the number of prime factors of $n$, counted
    with multiplicity. Then
    \begin{equation*}
        \#\{N-H<n\leq N : \Omega(n)>(1+\varepsilon)\log\log N\}
        \ll_{\varepsilon} \frac{H}{(\log N)^{\varepsilon^2/10}}.
    \end{equation*}
\end{lemma}

\begin{proof}
Put $z=1+\varepsilon/8$. The quantity we want to bound is certainly at most
\begin{equation*}
    \frac{1}{(\log N)^{(1+\varepsilon)\log z}}
    \sum_{N-H<n\leq N} z^{\Omega(n)}.
\end{equation*}
Take $\epsilon=3\log z/\log 2$, so that
$z^{\Omega(n)}\leq n^{\epsilon/3}$. Since
\[
    \epsilon\leq \frac{3\varepsilon}{8\log 2}
    <\frac{5\varepsilon}{8}<\frac18,
\]
we have $4\epsilon<5\varepsilon$. Put $X=N-H$. Since $H=o(N)$, we
have $X\sim N$ and $H\leq X$ for all sufficiently large $N$. Thus
\[
    X^{4\epsilon}\leq N^{5\varepsilon}\leq H\leq X.
\]
Applying Proposition \ref{prop:nair-tenenbaum} with $Q(x)=x$, $A=z$
and $B=1$, we obtain
\begin{align*}
    \sum_{N-H<n\leq N} z^{\Omega(n)}
    &\ll_{\varepsilon} H\prod_{p\leq X}\left(1-\frac{1}{p}\right)
      \sum_{m\leq X}\frac{z^{\Omega(m)}}{m}\\
    &\ll_{\varepsilon} H\prod_{p\leq X}
      \left(1-\frac{1}{p}\right)\left(1-\frac{z}{p}\right)^{-1}
      \ll_{\varepsilon} H(\log N)^{z-1},
\end{align*}
where we used $z<2$ and Mertens' estimate. Putting the two bounds
together gives that the quantity we want to bound is
$\ll_{\varepsilon} H/(\log N)^{(1+\varepsilon)\log z-(z-1)}$.
Finally, using $\log(1+t)\geq t-t^2/2$ for $t\geq 0$, we have
\[
    (1+\varepsilon)\log z-(z-1)
    \geq \varepsilon^2\left(\frac18-\frac{1+\varepsilon}{128}\right)
    \geq \frac{\varepsilon^2}{10},
\]
as desired.
\end{proof}

\begin{lemma}\label{lemma:unSmoothSquarefree}
    Let $\frac{1}{2}<\alpha<1$ be fixed, let $N$ be large enough in
    terms of $\alpha$, and let $N\geq H=H(N)\geq N^{\frac{4\alpha+1}{5}}$.
    Then there are $\gg_{\alpha} H$ square-free integers in the short
    interval $(N-H,N]$ such that $P^+(n)>N^{\alpha}$.
\end{lemma}

\begin{proof}
    Let $0<\theta<\frac{1}{5}$ be an exponent for which
    \cite[Theorem 1.1]{Pandey2024Squarefree} holds. Since
    \[
        \alpha+(1-\alpha)\theta
        <\alpha+\frac{1-\alpha}{5}=\frac{4\alpha+1}{5},
    \]
    we may fix $\alpha<\beta<1$ such that
    $\beta+(1-\beta)\theta<\frac{4\alpha+1}{5}$.

    The number of such $n$ is at least
    \begin{equation*}
        \sum_{N^{\alpha}<p\leq N^{\beta}}
        \sum_{\frac{N-H/2}{p}<m\leq\frac{N}{p}}\mu^2(m).
    \end{equation*}
    Indeed, $p>N^{\alpha}>\sqrt{N}$ implies that $m\leq N/p<p$.
    Thus, if $m$ is square-free, then $pm$ is square-free and
    $P^+(pm)=p$, so no integer is counted twice.

    For each prime in the outer sum, put $X=(N-H/2)/p$ and $Y=H/(2p)$.
    We have $X\asymp N/p$, $Y\leq X$, and
    \[
        H\geq N^{\frac{4\alpha+1}{5}}
        \geq N^{\theta}p^{1-\theta}.
    \]
    Hence $Y\gg X^{\theta}$. Also, $X\geq N^{1-\beta}/2$ tends to
    infinity uniformly in $p$. By \cite[Theorem 1.1]{Pandey2024Squarefree},
    it follows that
    \begin{equation*}
        \sum_{\frac{N-H/2}{p}<m\leq\frac{N}{p}}\mu^2(m)
        \gg \frac{H}{p}.
    \end{equation*}
    Summing over the primes and using Mertens' estimate, we obtain
    \begin{equation*}
        \sum_{N^{\alpha}<p\leq N^{\beta}}
        \sum_{\frac{N-H/2}{p}<m\leq\frac{N}{p}}\mu^2(m)
        \gg H\sum_{N^{\alpha}<p\leq N^{\beta}}\frac{1}{p}
        = H\left(\log\frac{\beta}{\alpha}+o(1)\right).
    \end{equation*}
    Since $\beta>\alpha$ is fixed in terms of $\alpha$, this gives the
    desired lower bound.
\end{proof}

\subsection{Probability results} We begin by adapting Proposition \ref{prop:quantCLT} to the complex case for the Steinhaus model.

\begin{prop}\label{prop:quantCLT-complex}
    Let $X_1,X_2,\ldots,X_n$ be a complex mean-zero martingale difference sequence
    such that
    \begin{equation*}
          \sum_{i=1}^n \mathbb E\lvert X_i\rvert^4
        + \mathbb E\left\lvert\sum_{i=1}^n\lvert X_i\rvert^2-1
          \right\rvert^2
        + \mathbb E\left\lvert\sum_{i=1}^n
          \left(X_i^2+\overline{X_i^2}\right)\right\rvert^2
        \leq \frac{1}{12}.
    \end{equation*}
    Let $S_n=\sum_{i=1}^n X_i$. Then, for all $x\in\mathbb R$,
    \begin{multline*}
        \left\lvert\mathbb P(\mathfrak T S_n\leq x)
        -\frac{1}{\sqrt{2\pi}}\int_{-\infty}^x
        e^{-\frac{t^2}{2}}\,\mathrm{d}t\right\rvert
        \\ \ll \frac{\left( \sum_{i=1}^n \mathbb E\lvert X_i\rvert^4
        + \mathbb E\left\lvert\sum_{i=1}^n\lvert X_i\rvert^2-1
          \right\rvert^2
        + \mathbb E\left\lvert\sum_{i=1}^n
          \left(X_i^2+\overline{X_i^2}\right)\right\rvert^2 \right)^{1/5}}{1+\lvert x\rvert^{\frac{16}{5}}},
    \end{multline*}
    where $\mathfrak T S_n$ denotes $\sqrt{2}$ times either the real
    or the imaginary part of $S_n$.
\end{prop}

\begin{proof}
    Fix either choice of $\mathfrak T$ and put $Y_i=\mathfrak T X_i$.
    Then $Y_1,\ldots,Y_n$ is a real martingale difference sequence
    and $\sum_{i=1}^n Y_i=\mathfrak T S_n$. We have
    $\lvert Y_i\rvert^4\leq 4\lvert X_i\rvert^4$ and
    \begin{equation*}
        \sum_{i=1}^n Y_i^2-1
        =\left(\sum_{i=1}^n\lvert X_i\rvert^2-1\right)
        \pm\frac12\sum_{i=1}^n\left(X_i^2+\overline{X_i^2}\right),
    \end{equation*}
    with the plus sign for the real part and the minus sign for the
    imaginary part. Hence, using
    $\lvert a+b\rvert^2\leq 2\lvert a\rvert^2+2\lvert b\rvert^2$,
    we obtain
    \begin{equation*}
        \sum_{i=1}^n\mathbb E\lvert Y_i\rvert^4
        +\mathbb E\left\lvert\sum_{i=1}^n Y_i^2-1\right\rvert^2
        \leq 4M_n\leq \frac{1}{3},
    \end{equation*}
    where $M_n$ denotes the expression on the left of the hypothesis of the proposition.
    Applying Proposition \ref{prop:quantCLT} to $Y_1,\ldots,Y_n$
    gives the desired bound.
\end{proof}

Next, we borrow a result from Billingsley's book \cite{Billingsley1999Convergence}, bootstrapping fourth moment information in short intervals to a deviation estimate for the maximum of a sum of random variables. 

\begin{lemma}[Theorem 10.2 in \cite{Billingsley1999Convergence}] \label{lemma:billings} Let $S_N = X_1 + X_2 + \ldots + X_N$ be a sum of random variables. Let $u_1, u_2, \ldots, u_n$ be non-negative real numbers such that 
\begin{equation*}
    \mathbb P(\lvert S_j - S_i\rvert\geq\lambda)\leq\frac{1}{\lambda^4}\left(\sum_{i<k\leq j} u_k\right)^2 
\end{equation*} for all $0\leq i<j\leq n$ (with $S_0 = 0$). Then 
\begin{equation*}
    \mathbb P\left(\max_{k\leq n} \lvert S_k\rvert\geq\lambda \right) \ll \frac{1}{\lambda^4}\left(\sum_{1\leq k\leq n} u_k\right)^2 .
\end{equation*}
    
\end{lemma}

To prove the almost sure lower bounds, we will actually show that our sums attain large values with large probability rather frequently. The next lemma establishes that this is sufficient.  

\begin{lemma}\label{lemma:asLbGeneral}
Let $(\mathcal A_N)_N$ be a collection of subsets of positive integers, and let $X_n$ for $n\in\mathcal A_N$ be random variables. Let $\mathcal L(X)$ and $\mathcal E(X)=o_{X\to \infty}(1)$ be positive functions such that $\mathcal L(x)\asymp \mathcal L(X)$ for all $X<x\leq X^2$, and 
\begin{equation*}
    \max_{X<N\leq X^2}\left\lvert\sum_{n\in\mathcal A_N} X_n \right\rvert\geq \mathcal L(X)
\end{equation*} for every large enough $X$ with probability $1-\mathcal E(X)$. Then almost surely there exists a sequence $N_l\to\infty$ such that 
\begin{equation}\label{eq:LargeL}
    \left\lvert \sum_{n\in\mathcal A_{N_l}} X_n \right\rvert \gg \mathcal L(N_l).
\end{equation}
\end{lemma}

\begin{proof}
   Begin with a sparse sequence $X_l$ satisfying $X_{l+1}>X_l^2$, say. For a fixed $l$, the probability that \eqref{eq:LargeL} fails for all $N\in (X_l, X_l^2]$ and a suitably small implied constant is at most $\mathcal E(X_l)$ by the condition of the proposition. Now we may choose a subsequence $X_{l_k}$ of $X_l$ so sparse depending on $\mathcal E$ such that $\sum_{k=1}^\infty \mathcal E\left(X_{l_k}\right)$ converges. Then by Borel-Cantelli, \eqref{eq:LargeL} fails for at most finitely many $l_k$. 
\end{proof}

\section{General Framework}\label{section:generalSetup}
We are going to deduce our results from the following general framework. Throughout the rest of this paper, let $f$ be a Rademacher (or Steinhaus) random multiplicative function. We are interested in non-empty subsets $\mathcal A_N\subseteq\{ 2, \ldots, N\}$ of square-free (or all) positive integers, sampled at various $N_1, N_2, \ldots, N_k$, with (some of) the following properties for any $1\leq l_1, \cdots , l_4\leq k$:
\begin{itemize}
    \item[(I)] There exists $\varepsilon_1>0$ such that the number of non-trivial\footnote{We call a solution \emph{non-trivial} if the $n_i$ are not equal in pairs.} solutions to $$n_1n_2n_3n_4=\square \text{ }(\text{or }n_1n_4=n_2n_3)$$ 
    with $n_i\in\mathcal A_{N_{l_i}}$ and $P^+(n_1)=P^+(n_2), P^+(n_3)=P^+(n_4)$ is $\leq \varepsilon_1\sqrt{\lvert\mathcal A_{N_{l_1}}\rvert\cdots \lvert\mathcal A_{N_{l_4}}\rvert}$.
     \item[(I')] There exists $\varepsilon_1'>0$ such that the number of non-trivial solutions to $n_1n_2n_3n_4=\square$ 
     (or $n_1n_4=n_2n_3$) 
    with $n_i\in\mathcal A_{N_{l_i}}$ and $P^+(n_1) = \cdots = P^+(n_4)$ is $\leq \varepsilon_1'\sqrt{\lvert\mathcal A_{N_{l_1}}\rvert\cdots \lvert\mathcal A_{N_{l_4}}\rvert}$. 
    
    (Note that $\varepsilon'_1$ may be taken to be at most $\varepsilon_1$ if (I) is satisfied.)

    \item[(II)] There exists $\varepsilon_2>0$ such that the number of $n_1\in\mathcal A_{N_{l_1}}$ and $n_2\in\mathcal A_{N_{l_2}}$ such that $P^+(n_1)=P^+(n_2)$ is $\leq \varepsilon_2\lvert\mathcal A_{N_{l_1}}\rvert\lvert\mathcal A_{N_{l_2}}\rvert$.

    \item[(III)] We either have: 
    \begin{itemize}
        \item[(III$_a$)] If $N_{l_1}\leq N_{l_2}$ then $\mathcal A_{N_{l_1}}\subseteq \mathcal A_{N_{l_2}}$, or
        \item[(III$_b$)] If $l_1\neq l_2$, then $\mathcal A_{N_{l_1}}\cap\mathcal A_{N_{l_2}} =\emptyset$. 
    \end{itemize} 

\end{itemize}

\begin{remark}
    Condition (III) is automatic if we begin with an infinite subset of square-free (or all) positive integers $\mathcal A$, and let $\mathcal A_N = \mathcal A \cap(N-H, N]$ for some $H=H(N)\to\infty$ and $H=o(N)$, as long as the scales $N_1, N_2, \ldots, N_k$ are sufficiently separated. 
\end{remark}
Let
\begin{equation}\label{eq:S_Ndef}
    S_N = \frac{1}{\sqrt{\lvert \mathcal A_N\rvert}}\sum_{n\in\mathcal A_N} \mathfrak Tf(n) = \sum_{p} \frac{1}{\sqrt{\lvert \mathcal A_N\rvert}} \sum_{\substack{n\in\mathcal A_N \\ P^+(n) = p}} \mathfrak Tf(n) =: \sum_{p} M_{p, N},
\end{equation}
where here and throughout the paper $\mathfrak T$ denotes simply the identity in the Rademacher case, and $\sqrt2$ times either the real or imaginary part in the Steinhaus case. 

\subsection{Approximate joint Gaussianity}
For any sequence of integers $N_1< N_2< \cdots< N_k$ and real coefficients $c_1, c_2, \ldots, c_k$ with $\sum_{l=1}^k \lvert c_l\rvert^2 = 1$, the sum 
\begin{equation*}
    c_1S_{N_1} + c_2S_{N_2} + \cdots + c_kS_{N_k} = \sum_p \sum_{l=1}^k c_lM_{p, N_l}
\end{equation*} is a martingale, with $\sum_{l=1}^k c_lM_{p, N_l}$ being a martingale difference sequence indexed by primes $p$.

Applying Proposition \ref{prop:quantCLT} (or Proposition \ref{prop:quantCLT-complex}) to $\sum_{l=1}^k c_lM_{p, N_l}$ leads us to bounding the quantities 
\begin{equation*}
    A := \sum_{p} \mathbb E\left\lvert\sum_{l=1}^k c_lM_{p, N_l}\right\rvert^4 \hspace{1cm}\text{and}\hspace{1cm} B:= \mathbb E\left\lvert\sum_{p} \left\lvert \sum_{l=1}^k c_lM_{p, N_l}\right\rvert^2 - 1\right\rvert^2.
\end{equation*}
In the Steinhaus case, for our particular martingale, the last term from Proposition \ref{prop:quantCLT-complex} is $\ll A$. This is because upon expanding
\begin{equation*}
    \mathbb E \left\lvert\sum_p \left(\left(\sum_{l=1}^k c_lM_{p, N_l}\right)^2 + \left(\sum_{l=1}^k \overline{c_lM_{p, N_l}}\right)^2\right) \right\rvert^2,
\end{equation*}
for all primes $p\neq q$ we have $$\mathbb EM_{p, N_{l_1}}^2M_{q, N_{l_2}}^2=\mathbb E\overline{M_{p, N_{l_1}}^2}M_{q, N_{l_2}}^2 = \mathbb EM_{p, N_{l_1}}^2\overline{M_{q, N_{l_2}}^2} = \mathbb E\overline{M_{p, N_{l_1}}^2M_{q, N_{l_2}}^2}=0,$$ whereas for the primes $p=q$ the surviving terms are (note that $\mathbb E\left(\sum_{l=1}^k c_l M_{p, N_l}\right)^4 =$ \linebreak $\mathbb E\overline{\left(\sum_{l=1}^k c_l M_{p, N_l}\right)}^4 =  0$ since $\mathbb Ef(p)^4 = 0$)
\begin{equation*}
    \ll \sum_{p} \mathbb E\left(\sum_{l=1}^k c_lM_{p, N_l}\right)^2\overline{\left(\sum_{l=1}^k c_l M_{p, N_l}\right)^2} = A.
\end{equation*}

\subsubsection{Bounding $A$} We have
\begin{equation*}
    A = \sum_p\sum_{1\leq l_i\leq k} c_{l_1}c_{l_2}\overline{c_{l_3}}\overline{c}_{l_4}\mathbb E M_{p, N_{l_1}}M_{p, N_{l_2}}\overline{M}_{p, N_{l_3}}\overline{M}_{p, N_{l_4}}. 
\end{equation*}
Further, 
\begin{equation*}
    \mathbb E M_{p, N_{l_1}}M_{p, N_{l_2}}\overline{M}_{p, N_{l_3}}\overline{M}_{p, N_{l_4}} = \frac{1}{\sqrt{\lvert \mathcal A_{N_{l_1}}\rvert\cdots \lvert \mathcal A_{N_{l_4}}\rvert}}\sum_{\substack{n_i\in\mathcal A_{N_{l_i}} \\ P^+(n_i) = p}} \mathbb Ef(n_1)f(n_2)\overline{f(n_3)} \overline{f(n_4)}.
\end{equation*}

Note that $\mathbb Ef(n_1)\cdots \overline{f(n_4)}$ is non-zero only if $n_1n_2n_3n_4$ is a perfect square (or $n_1n_2=n_3n_4$), in which case the expectation is equal to $1$. Therefore, by moving the sum over $p$ inside, we have 
\begin{equation*}
    A = \sum_{1\leq l_i\leq k} \frac{c_{l_1}}{\sqrt{\lvert\mathcal A_{N_{l_1}}\rvert}}\cdots \frac{\overline{c}_{l_4}}{\sqrt{\lvert\mathcal A_{N_{l_4}}}\rvert}\sum_{\substack{n_i\in\mathcal A_{N_{l_i}} \\ n_1n_2n_3n_4 = \square \text{ }(\text{or } n_1n_2=n_3n_4) \\ P^+ (n_i) = P^+ (n_j) \text{ }\forall i, j }}1.
\end{equation*}
The contribution from the trivial solutions (note that by condition (III), $n_1=n_2$ forces one of $\mathcal A_{N_{l_1}}, \mathcal A_{N_{l_2}}$ to be a subset of the other, and similarly for the other variables) is 
\begin{equation*}
\ll  \sum_{\substack{1\leq l_i\leq k }} \frac{\lvert c_{l_1}\rvert\cdots \lvert c_{l_4}\rvert}{\sqrt{\lvert \mathcal A_{N_{l_1}}\rvert\cdots \lvert\mathcal A_{N_{l_4}}\rvert}}\sum_{\substack{n_1\in\mathcal A_{N_{l_1}}\cap \mathcal A_{N_{l_2}} \\ n_3\in\mathcal A_{N_{l_3}}\cap \mathcal A_{N_{l_4}} \\ P^+(n_1) = P^+(n_3)}} 1\ll \varepsilon_2k^2, 
\end{equation*}
by condition (II). Here we used that if, say $\mathcal A_{N_{l_1}}\cap \mathcal A_{N_{l_2}} = \mathcal A_{N_{l_1}}$, then certainly $\lvert\mathcal A_{N_{l_2}}\rvert\geq\lvert\mathcal A_{N_{l_1}}\rvert$, and then applied the Cauchy-Schwarz inequality for the sum over the $l_i$. Note that if (III$_b$) is satisfied, then $n_1=n_2$ forces $l_1=l_2$ (and similarly for the other variables), in which case the contribution from the trivial solutions would be $\ll\varepsilon_2 \sum_{1\leq l_1, l_3\leq k} \lvert c_{l_1}\rvert^2\lvert c_{l_3}\rvert^2\ll \varepsilon_2.$

The contribution from non-trivial solutions is
\begin{equation*}
    \ll\varepsilon_1' \sum_{1\leq l_i\leq k} \lvert c_{l_1}\rvert\cdots \lvert c_{l_4}\rvert \ll \varepsilon_1'k^2 
\end{equation*}
using $\sum_{l=1}^k \lvert c_{l}\rvert\ll\sqrt k$ by Cauchy-Schwarz and condition (I').

We conclude that
\begin{equation}\label{eq:Abound}
    A \ll k^2(\varepsilon_1'  + \varepsilon_2),
\end{equation}
or 
\begin{equation}\label{eq:AboundDisjoint}
    A\ll k^2\varepsilon_1' + \varepsilon_2
\end{equation}
if (III$_b$) is satisfied. 

\subsubsection{Bounding $B$} We have
\begin{multline*}
    B = 1+ \sum_{p,q} \sum_{1\leq l_1,l_2,l_3,l_4\leq k}
    \frac{c_{l_1}}{\sqrt{\lvert \mathcal A_{N_{l_1}}\rvert}}\cdots
    \frac{c_{l_4}}{\sqrt{\lvert \mathcal A_{N_{l_4}}\rvert}}
    \sum_{\substack{n_i\in\mathcal A_{N_{l_i}} \\
    n_1n_2n_3n_4=\square \text{ }(\text{or }n_1n_4=n_2n_3) \\
    P^+(n_1)=P^+(n_2)=p \\ P^+(n_3)=P^+(n_4)=q}} 1 \\
    -2\sum_{1\leq l_1,l_2\leq k}
    \frac{c_{l_1}}{\sqrt{\lvert \mathcal A_{N_{l_1}}\rvert}}
    \frac{c_{l_2}}{\sqrt{\lvert \mathcal A_{N_{l_2}}\rvert}}
    \lvert \mathcal A_{N_{l_1}}\cap\mathcal A_{N_{l_2}}\rvert.
\end{multline*}
Here we used that $n_1n_2=\square$ if and only if $n_1=n_2$ in the Rademacher case, since $n_1$ and $n_2$ are square-free.

In the first sum, the contribution from the trivial solutions $n_1=n_2$, $n_3=n_4$ is
\begin{equation*}
    \left(\sum_{1\leq l_1,l_2\leq k}
    \frac{c_{l_1}}{\sqrt{\lvert \mathcal A_{N_{l_1}}\rvert}}
    \frac{c_{l_2}}{\sqrt{\lvert \mathcal A_{N_{l_2}}\rvert}}
    \lvert \mathcal A_{N_{l_1}}\cap\mathcal A_{N_{l_2}}\rvert\right)^2,
\end{equation*}
since, after moving the sums over $p,q$ inside, $n_1=n_2$ and $n_3=n_4$ may be chosen independently. Combining this contribution with $1$ and the last sum in the expression for $B$, and using $\sum_{l=1}^k c_l^2=1$, we obtain exactly
\begin{equation*}
    \left(2\sum_{1\leq l_1<l_2\leq k}
    \frac{c_{l_1}}{\sqrt{\lvert \mathcal A_{N_{l_1}}\rvert}}
    \frac{c_{l_2}}{\sqrt{\lvert \mathcal A_{N_{l_2}}\rvert}}
    \lvert \mathcal A_{N_{l_1}}\cap\mathcal A_{N_{l_2}}\rvert\right)^2.
\end{equation*}
If (III$_a$) is satisfied, then $\mathcal A_{N_{l_1}}\subseteq\mathcal A_{N_{l_2}}$ for $l_1<l_2$, so this is
\begin{equation*}
    \left(2\sum_{1\leq l_1<l_2\leq k} c_{l_1}c_{l_2}
    \sqrt{\frac{\lvert \mathcal A_{N_{l_1}}\rvert}
    {\lvert \mathcal A_{N_{l_2}}\rvert}}\right)^2
    \ll \left(k\max_{1\leq l_1<l_2\leq k}
    \sqrt{\frac{\lvert \mathcal A_{N_{l_1}}\rvert}
    {\lvert \mathcal A_{N_{l_2}}\rvert}}\right)^2,
\end{equation*}
where we used $\sum_{l=1}^k\lvert c_l\rvert\leq\sqrt{k}$ by Cauchy-Schwarz. If (III$_b$) is satisfied, the contribution is exactly zero, since the intersections are empty for $l_1\neq l_2$.

The contribution from non-trivial solutions, using (I), is in either case
\begin{equation*}
    \ll \varepsilon_1\sum_{1\leq l_1,l_2,l_3,l_4\leq k}
    \lvert c_{l_1}\rvert\cdots\lvert c_{l_4}\rvert
    \ll \varepsilon_1 k^2.
\end{equation*}
It remains to bound the trivial solutions $n_1=n_3$, $n_2=n_4$, with $n_1\neq n_2$, and, in the Rademacher case, also $n_1=n_4$, $n_2=n_3$, with $n_1\neq n_2$. The solutions with $n_1=n_2$ have already been counted. The two remaining pairings give the same bound by interchanging $n_3,n_4$ and $l_3,l_4$. Thus their contribution is
\begin{equation*}
    \ll \sum_{1\leq l_1,l_2,l_3,l_4\leq k}
    \frac{\lvert c_{l_1}\rvert}{\sqrt{\lvert \mathcal A_{N_{l_1}}\rvert}}
    \cdots
    \frac{\lvert c_{l_4}\rvert}{\sqrt{\lvert \mathcal A_{N_{l_4}}\rvert}}
    \sum_{\substack{n_i\in\mathcal A_{N_{l_i}} \\
    n_1=n_3,\ n_2=n_4,\ n_1\neq n_2 \\
    P^+(n_1)=P^+(n_2)}} 1.
\end{equation*}
As before, $n_1=n_3$ implies $n_1\in\mathcal A_{N_{l_1}}\cap\mathcal A_{N_{l_3}}$, and similarly for $n_2=n_4$. Applying condition (II) first to the sets indexed by $l_1,l_2$ and then to those indexed by $l_3,l_4$, we see that the inner sum is at most
\begin{align*}
    &\varepsilon_2\min\left\{
    \lvert\mathcal A_{N_{l_1}}\rvert\lvert\mathcal A_{N_{l_2}}\rvert,
    \lvert\mathcal A_{N_{l_3}}\rvert\lvert\mathcal A_{N_{l_4}}\rvert
    \right\}\leq \varepsilon_2
    \sqrt{\lvert\mathcal A_{N_{l_1}}\rvert\cdots
    \lvert\mathcal A_{N_{l_4}}\rvert}.
\end{align*}
Thus the contribution from these solutions is $\ll\varepsilon_2 k^2$ by Cauchy-Schwarz. If (III$_b$) is satisfied, then $n_1=n_3$ and $n_2=n_4$ force $l_1=l_3$ and $l_2=l_4$, respectively. In this case the bound improves to
\begin{equation*}
    \ll \varepsilon_2\sum_{1\leq l_1,l_2\leq k}
    \lvert c_{l_1}\rvert^2\lvert c_{l_2}\rvert^2
    =\varepsilon_2.
\end{equation*}
We conclude that
\begin{equation}\label{eq:Bbound}
    B \ll k^2(\varepsilon_1+\varepsilon_2)
    +\left(k\max_{1\leq l_1<l_2\leq k}
    \sqrt{\frac{\lvert \mathcal A_{N_{l_1}}\rvert}
    {\lvert \mathcal A_{N_{l_2}}\rvert}}\right)^2
\end{equation}
if (III$_a$) is satisfied, or
\begin{equation}\label{eq:BboundDisjoint}
    B \ll k^2\varepsilon_1+\varepsilon_2
\end{equation}
if (III$_b$) is satisfied.

Combining \eqref{eq:Abound}, \eqref{eq:AboundDisjoint}, \eqref{eq:Bbound}, \eqref{eq:BboundDisjoint} and using Propositions \ref{prop:quantCLT} and \ref{prop:quantCLT-complex},
we obtain the following approximate joint Gaussianity result. 

\begin{prop}\label{prop:approxJointGaussian}
    Let $f$ be a Rademacher or Steinhaus random multiplicative function, and suppose the $\mathcal A_{N_l}$ for $1\leq l\leq k$  satisfy conditions (I), (I'), (II) and (III), with the values $\varepsilon_1, \varepsilon_1', \varepsilon_2$, respectively. Then, for any real coefficients $c_1, c_2, \ldots, c_k$ with $\sum_{l=1}^k \lvert c_l\rvert^2 = 1$, we have
    \begin{multline*}
\left\lvert\mathbb P\left(\sum_{l=1}^k c_l S_{N_{l}}\leq x \right) - \frac{1}{\sqrt{2\pi}}\int_{-\infty}^x e^{-\frac{t^2}{2}}\emph{d}t\right\rvert \\ \ll \frac{1}{1+\lvert x\rvert^{\frac{16}{5}}}\left(   \varepsilon_2 + k^2(\varepsilon_1 + \varepsilon_1' ) + (1-1_{\text{(III$_b$)}})\left(k^2\varepsilon_2 + \left(k \max_{1\leq l_1<l_2\leq k} \sqrt{\frac{\lvert \mathcal A_{N_{l_1}}\rvert}{\lvert \mathcal A_{N_{l_2}}\rvert}}\right)^2\right) \right)^{\frac{1}{5}},
 \end{multline*}
 where $1_{\text{(III$_b$)}}$ denotes the indicator function of (III$_b$), so long as the expression under the fifth root in the right-hand side is at most a small absolute constant.
\end{prop}

In certain cases, it will be necessary to have an analogous result with an error in the quantitative central limit theorem independent of $\varepsilon_1$. For this, we introduce an additional assumption:
\begin{itemize}
    \item[(IV)] Let $P$ be a parameter and suppose that the sets $\mathcal A_{N_l}$ ($1\leq l\leq k$) consist only of integers divisible by exactly one prime $p>P$ (that may depend on $n$), and $p^2\nmid n$.
\end{itemize}

The advantage of introducing condition (IV) and Proposition \ref{prop: complicatedapproxJointG} is that the parameters $\varepsilon_1'$ and $\varepsilon_2$ can be taken to be much smaller than before (depending on the size of $P$). This leads to a very precise approximation of the conditional probabilities over primes bigger than $P$, at least over most realizations of the $f(p)$ for $p\leq P$. 

Under the assumption (IV), we condition on all the primes $\leq P$ and note that 
\begin{equation*}
    \frac{1}{V_c}\sum_{p>P}\sum_{l=1}^k c_l M_{p, N_l}
\end{equation*} is a martingale (in fact it is a weighted sum of the independent random variables $f(p)$ with $p>P$), where \begin{equation}\label{eq:V_c}
    V_c =\sqrt{\widetilde{\mathbb E}\left(\sum_{p>P}\sum_{l=1}^k c_l M_{p, N_l} \right)^2 } = \sqrt{\sum_{1\leq l_1, l_2\leq k}c_{l_1}c_{l_2}\widetilde{\mathbb E}S_{N_{l_1}}S_{N_{l_2}}}.
\end{equation} Here and throughout $\widetilde{\mathbb P}$ and $\widetilde{\mathbb E}$ denote conditional probability and conditional expectation, respectively, being taken only over primes $>P$. 

\begin{prop}\label{prop: complicatedapproxJointG}
    Let $f$ be a Rademacher (or Steinhaus) random multiplicative function,
    and suppose the $\mathcal A_{N_l}$ for $1\leq l\leq k$ satisfy
    conditions (I'), (II), (III), and (IV), with the values
    $\varepsilon_1',\varepsilon_2$, and $P$, respectively. Suppose that
    with probability $1-O(\eta)$ over the $f(p)$ for primes $p\leq P$,
    we have $1/2\leq V_c^2\leq 2$ for all real coefficients
    $c_1,\ldots,c_k$ with $\sum_{l=1}^k\lvert c_l\rvert^2=1$.
    Then, with probability
    \begin{equation}\label{eq:bigProb}
        1-O\left(\eta+k\sqrt{\varepsilon_1'+\varepsilon_2}\right)
    \end{equation}
    over the $f(p)$ for primes $p\leq P$, we have
    \begin{equation*}
        \left\lvert\widetilde{\mathbb P}\left(
        \sum_{p>P}\sum_{l=1}^k c_l M_{p,N_l}\leq x\right)
        -\frac{1}{\sqrt{2\pi V_c^2}}\int_{-\infty}^x
        e^{-\frac{t^2}{2V_c^2}}\,\mathrm{d}t\right\rvert
        \ll \frac{\left(k\sqrt{\varepsilon_1'+\varepsilon_2}\right)^{1/5}}
        {1+\lvert x\rvert^{\frac{16}{5}}}
    \end{equation*}
    for all such coefficients and all $x\in\mathbb R$, provided that
    $k\sqrt{\varepsilon_1'+\varepsilon_2}$ is sufficiently small.
    If condition (I) is also satisfied, we may take
    \begin{equation*}
        \eta=k^2(\varepsilon_1+\varepsilon_2)
        +(1-1_{\text{(III$_b$)}})
        \left(k\max_{1\leq l_1<l_2\leq k}
        \sqrt{\frac{\lvert\mathcal A_{N_{l_1}}\rvert}
        {\lvert\mathcal A_{N_{l_2}}\rvert}}\right)^2.
    \end{equation*}
\end{prop}

\begin{proof}
    By Cauchy-Schwarz and the bound for $A$ at a single scale, we have
    \begin{equation*}
        \mathbb E\widetilde{\mathbb E}\sum_{p>P}
        \left(\sum_{l=1}^k\lvert M_{p,N_l}\rvert^2\right)^2
        \leq k\sum_{l=1}^k\mathbb E\sum_{p>P}\lvert M_{p,N_l}\rvert^4
        \ll k^2(\varepsilon_1'+\varepsilon_2).
    \end{equation*}
    Thus, by Markov's inequality and the hypothesis on $V_c$, with
    probability as in \eqref{eq:bigProb} we have $1/2\leq V_c^2\leq 2$
    and
    \begin{equation*}
        \widetilde{\mathbb E}\sum_{p>P}
        \left\lvert\sum_{l=1}^k c_l M_{p,N_l}\right\rvert^4
        \leq k\sqrt{\varepsilon_1'+\varepsilon_2}
    \end{equation*}
    for all the stated choices of the coefficients.

    By (IV), after conditioning on the $f(p)$ for primes $p\leq P$,
    the sums $\sum_{l=1}^k c_l M_{p,N_l}$ are independent real
    mean-zero random variables as $p$ varies, where in the Steinhaus
    case we take the chosen real or imaginary part. Hence, by
    the definition of $V_c$ and independence,
    \begin{align*}
        \widetilde{\mathbb E}\left\lvert
        \frac{1}{V_c^2}\sum_{p>P}
        \left(\sum_{l=1}^k c_l M_{p,N_l}\right)^2-1\right\rvert^2
        &=\frac{1}{V_c^4}\sum_{p>P}
        \widetilde{\operatorname{Var}}\left(
        \left(\sum_{l=1}^k c_l M_{p,N_l}\right)^2\right)\\
        &\leq\frac{1}{V_c^4}\widetilde{\mathbb E}\sum_{p>P}
        \left\lvert\sum_{l=1}^k c_l M_{p,N_l}\right\rvert^4.
    \end{align*}
    Applying Proposition \ref{prop:quantCLT} to the sum divided by
    $V_c$, at the point $x/V_c$, gives the desired bound since
    $1/2\leq V_c^2\leq 2$.

    Finally, if (I) is satisfied, applying the bound for $B$ to single
    coordinates and to the sums and differences of pairs of coordinates,
    and using conditional Jensen's inequality, gives
    \begin{equation*}
        \sum_{1\leq l_1,l_2\leq k}
        \mathbb E\left\lvert
        \widetilde{\mathbb E}\left(S_{N_{l_1}}S_{N_{l_2}}\right)
        -1_{l_1=l_2}\right\rvert^2\ll\eta
    \end{equation*}
    with $\eta$ as stated. By Cauchy-Schwarz, the corresponding sum
    of squares bounds $\lvert V_c^2-1\rvert^2$ for every choice of
    the coefficients. Markov's inequality now gives the claimed
    probability for $1/2\leq V_c^2\leq 2$.
\end{proof}

\subsection{Multivariate Central Limit Theorem} Finally, we use Proposition \ref{prop:multivarCLT} to bootstrap Propositions \ref{prop:approxJointGaussian} and \ref{prop: complicatedapproxJointG} to quantitative multivariate central limit theorems.     

\begin{prop}\label{prop:MCLT}
Let $f$ be a Rademacher (or Steinhaus) random multiplicative function, and suppose the $\mathcal A_{N_l}$ for $1\leq l\leq k$ satisfy conditions (I), (I'), (II) and (III), with the values $\varepsilon_1, \varepsilon_1', \varepsilon_2$, respectively. Further, let $X=(S_{N_1}, S_{N_2}, \ldots, S_{N_k})$ (recall \eqref{eq:S_Ndef} for the definition of $S_N$) and let $Y$ be a standard multivariate Gaussian random vector with independent coordinates. Then
\begin{multline}
    \sup_{\|u\|_{\emph{Lip}}\leq 1}
    \lvert \mathbb E u(X)-\mathbb E u(Y)\rvert
    \ll k^{\frac{1}{2}}\bigg(\varepsilon_2+k^2(\varepsilon_1+\varepsilon_1')\\
    +(1-1_{\text{(III$_b$)}})\left(k^2\varepsilon_2+
    \left(k\max_{1\leq l_1<l_2\leq k}
    \sqrt{\frac{\lvert\mathcal A_{N_{l_1}}\rvert}
    {\lvert\mathcal A_{N_{l_2}}\rvert}}\right)^2\right)
    \bigg)^{\frac{1}{5(k+1)}},
\end{multline}
provided the error is smaller than some absolute constant. 

If instead the hypotheses of Proposition \ref{prop: complicatedapproxJointG} are satisfied, then with probability as in \eqref{eq:bigProb} over the $f(p)$ for primes $p\leq P$, we have
\begin{equation}
    \sup_{\|u\|_{\emph{Lip}}\leq 1}
    \lvert \widetilde{\mathbb E}u(X)-{\mathbb E}u(Y')\rvert
    \ll k^{\frac{1}{2}}
    \left(k\sqrt{\varepsilon_1'+\varepsilon_2}\right)^{\frac{1}{5(k+1)}},
\end{equation}
where, conditional on the $f(p)$ for primes $p\leq P$, $Y'$ is a mean-zero multivariate Gaussian random vector with
${\mathbb E}Y'_iY'_j=\widetilde{\mathbb E}S_{N_i}S_{N_j}$ for $1\leq i,j\leq k$.
\end{prop}

\begin{proof}
We apply Proposition \ref{prop:multivarCLT} with $p=q=2$. For the first bound, note that $\mathbb E\lvert X\rvert^2=\mathbb E\lvert Y\rvert^2=k$. Integrating the bound in Proposition \ref{prop:approxJointGaussian} over $x$ gives the desired result, since $(1+\lvert x\rvert^{16/5})^{-1}$ is integrable.
For the second bound, we work on the event from Proposition \ref{prop: complicatedapproxJointG}, on which its conclusion and $1/2\leq V_c^2\leq 2$ hold for all unit vectors $c$ simultaneously. Taking $c$ to be each of the coordinate vectors gives
\begin{equation*}
    \widetilde{\mathbb E}\lvert X\rvert^2
    ={\mathbb E}\lvert Y'\rvert^2
    =\sum_{l=1}^k\widetilde{\mathbb E}S_{N_l}^2\leq 2k.
\end{equation*}
Moreover, conditional on the $f(p)$ for primes $p\leq P$, the sum $\sum_{l=1}^k c_lY_l'$ is Gaussian with mean $0$ and variance $V_c^2$. We may therefore integrate the bound in Proposition \ref{prop: complicatedapproxJointG} over $x$ and take the supremum over all unit vectors $c$ on this same event. Applying Proposition \ref{prop:multivarCLT} to the conditional distributions gives the desired bound.
\end{proof}

\begin{remark}
    As we remarked before, assumption (IV) places us in a position where 
    \begin{equation*}
         \frac{1}{V_c}\sum_{p>P}\sum_{l=1}^k c_l M_{p, N_l}
    \end{equation*}
    is conditionally a weighted sum of independent random variables. Therefore, instead of proving Proposition \ref{prop: complicatedapproxJointG} and bootstrapping it to Proposition \ref{prop:MCLT}, we could have opted to use the Normal Approximation Result 1 from the work of Harper \cite{Harper2023_ASLF}, which yields a stronger result in terms of the dependence on $k$. However, our argument only utilizes the martingale structure and may potentially be useful in other situations with a weaker assumption than (IV). 
\end{remark}

Proposition \ref{prop:MCLT} may then be used to estimate the probability that $\max_{1\leq l\leq k} S_{N_l}\leq t$ in terms of the purely Gaussian counterpart.  Previously we had proved a version of the following corollary by applying Proposition \ref{prop:MCLT} to $k$-dimensional random vectors, which incurs some loss in the number of scales $k$. In contrast, the present version of the corollary only utilizes the two-dimensional case of Proposition \ref{prop:MCLT} by observing that if $\max_{1\leq l\leq k}S_{N_l}\leq t$, then $\sum_{l=1}^k s(S_{N_l}) = 0$, where the function $s$ is defined in the proof below and roughly captures the property of being bigger than $t$. Applying Markov's inequality for the second moment to the latter event yields an expression for which the two-dimensional central limit theorem is sufficient. The number of scales $k$ only plays a role in the distance between the scales, as well as in controlling the roughly $k^2$ conditional covariances simultaneously.

\begin{cor}\label{cor:pairwiseNormalComparison}
    Adopt the notation of Proposition \ref{prop:MCLT}, let $k\geq 2$ and
    $t\geq 0$, and put
    \begin{equation*}
        p_t=\frac{1}{\sqrt{2\pi}}\int_t^\infty
        e^{-\frac{x^2}{2}}\,\mathrm{d}x.
    \end{equation*}
    Under the hypotheses of the first part of Proposition \ref{prop:MCLT},
    we have
    \begin{equation*}
        \mathbb P\left(\max_{1\leq l\leq k}S_{N_l}\leq t\right)
        \ll \frac{1}{kp_t}+\frac{1+t}{p_t^2}
        \bigg(\varepsilon_1+\varepsilon_1'+\varepsilon_2
        +(1-1_{\text{(III$_b$)}})
        \max_{1\leq l_1<l_2\leq k}
        \frac{\lvert\mathcal A_{N_{l_1}}\rvert}
        {\lvert\mathcal A_{N_{l_2}}\rvert}\bigg)^{\frac{1}{15}}.
    \end{equation*}
    In the conditional case, on any common event on which
    $1/2\leq\widetilde{\mathbb E}S_{N_l}^2\leq 2$ for all $l$ and
    the conditional bound in Proposition \ref{prop:MCLT} holds in dimension
    two for every pair of coordinates, we have
    \begin{equation*}
        \widetilde{\mathbb P}\left(\max_{1\leq l\leq k}
        \frac{S_{N_l}}{\sqrt{\widetilde{\mathbb E}S_{N_l}^2}}
        \leq t\right)
        \ll \frac{1}{kp_t}
        +(1+t)^2\max_{1\leq l_1<l_2\leq k}
        \frac{\left\lvert\widetilde{\mathbb E}
        (S_{N_{l_1}}S_{N_{l_2}})\right\rvert}
        {\sqrt{\widetilde{\mathbb E}S_{N_{l_1}}^2\,
        \widetilde{\mathbb E}S_{N_{l_2}}^2}}
        +\frac{1+t}{p_t^2}
        \left(\varepsilon_1'+\varepsilon_2\right)^{\frac{1}{30}}.
    \end{equation*}
\end{cor}

\begin{proof}
    Let $s(x)$ be zero for $x\leq t$, one for $x\geq t+(1+t)^{-1}$,
    and linear in between. By the Gaussian tail estimate,
    \[
        p_t\geq\mathbb E s(Y_1)\geq p_{t+(1+t)^{-1}}\gg p_t.
    \]
    If $\max_{1\leq l\leq k}S_{N_l}\leq t$, then
    $\sum_{l=1}^k s(S_{N_l})=0$. Hence
    \begin{equation*}
        \mathbb P\left(\max_{1\leq l\leq k}S_{N_l}\leq t\right)
        \ll \frac{1}{p_t^2}\mathbb E\left(
        \frac{1}{k}\sum_{l=1}^k s(S_{N_l})-\mathbb E s(Y_1)
        \right)^2.
    \end{equation*}
    Expanding the square, we apply Proposition \ref{prop:MCLT} in dimension
    two to each pair. The functions being compared have Lipschitz
    constant $\ll 1+t$, since $0\leq s\leq 1$. After replacing the
    $S_{N_l}$ by the $Y_l$, the diagonal contribution to the second
    moment is $\ll p_t/k$ and the off-diagonal contribution is zero
    by independence. This gives the first bound.

    The same argument applies conditionally after dividing the $l$-th
    coordinate of both vectors by $\sqrt{\widetilde{\mathbb E}S_{N_l}^2}$.
    The Gaussian coordinates are then standard normal, so the centering
    is still $\mathbb E s(Y_1)$. It remains to bound their off-diagonal
    contribution. We may assume that the middle term on the right-hand
    side of the conditional bound is at most $1/2$, since otherwise
    the bound is trivial. For thresholds $a,b\geq t$, the bivariate
    Gaussian density at $(a,b)$ is then $\ll e^{-t^2}$ at every
    correlation between zero and the correlation of the two coordinates.
    Integrating this density with respect to the correlation, and
    averaging $a,b$ over $[t,t+(1+t)^{-1}]$, bounds the covariance
    of the two smoothed indicators by
    \begin{equation*}
        \ll \frac{\left\lvert\widetilde{\mathbb E}
        (S_{N_{l_1}}S_{N_{l_2}})\right\rvert}
        {\sqrt{\widetilde{\mathbb E}S_{N_{l_1}}^2\,
        \widetilde{\mathbb E}S_{N_{l_2}}^2}}\,(1+t)^2p_t^2,
    \end{equation*}
    where we used $e^{-t^2}\ll (1+t)^2p_t^2$. The conditional bound
    now follows from Proposition \ref{prop:MCLT} applied to each pair on
    the stated common event.
\end{proof}

\subsection{Slow variation} Here we shall prove a general slow variation property of sums of random multiplicative functions, which is also applicable in short intervals. This generalizes a result found in the work of Lau, Tenenbaum and Wu \cite{LauTenenbaumWu2013MeanValuesRMF}, making use of Lemma \ref{lemma:billings}. In this paper, we will use this proposition in proving the almost sure upper bound in the case of polynomial images, but it is relevant in proving almost sure upper bounds in other cases too. 
\begin{prop}\label{prop:slowVar}
Let $f$ be a Rademacher (or Steinhaus) random multiplicative function.
Let $\mathcal A$ be an infinite subset of the square-free (or all)
positive integers, and suppose that $\mathcal A_N=\mathcal A\cap (N-H,N]$
for some function $H=H(N)\to\infty$ such that $H\leq N$ and $N-H$ is
eventually non-decreasing. For any increasing sequence $(N_l)_{l=1}^\infty$
such that $N_1$ is large enough, writing $H_l=H(N_l)$, we have for all
$l\geq 1$ and $\lambda>0$,
\begin{multline*}
    \mathbb P\left(\max_{N_l<N\leq N_{l+1}}
    \left\lvert\sum_{n\in\mathcal A_N}f(n)
    -\sum_{n\in\mathcal A_{N_l}}f(n)\right\rvert\geq\lambda\right)\\
    \ll \frac{1}{\lambda^4}\left(
    \left(\sum_{\substack{N_l-H_l<n\leq N_{l+1}-H_{l+1}\\
    n\in\mathcal A}}\tau_3(n)\right)^2+
    \left(\sum_{\substack{N_l<n\leq N_{l+1}\\
    n\in\mathcal A}}\tau_3(n)\right)^2\right).
\end{multline*}
\end{prop}

\begin{proof}
Let $M_l=N_l-H_l$. For $N_l<N\leq N_{l+1}$, we have
$M_l\leq N-H\leq M_{l+1}$. Writing
\begin{equation*}
    \sum_{n\in\mathcal A_N}f(n)-\sum_{n\in\mathcal A_{N_l}}f(n)
    =-\sum_{\substack{M_l<n\leq N-H\\n\in\mathcal A}}f(n)
    +\sum_{\substack{N_l<n\leq N\\n\in\mathcal A}}f(n),
\end{equation*}
it follows by the triangle inequality that if the maximum in the
statement is at least $\lambda$, then
\begin{equation*}
    \max_{M_l\leq w\leq M_{l+1}}
    \left\lvert\sum_{\substack{M_l<n\leq w\\n\in\mathcal A}}f(n)\right\rvert
    \geq\frac{\lambda}{2}
    \hspace{0.5cm}\text{or}\hspace{0.5cm}
    \max_{N_l\leq w\leq N_{l+1}}
    \left\lvert\sum_{\substack{N_l<n\leq w\\n\in\mathcal A}}f(n)\right\rvert
    \geq\frac{\lambda}{2}.
\end{equation*}
Note that for every $u<v$ and $\lambda>0$ we have
\begin{equation*}
    \mathbb P\left(\left\lvert
    \sum_{\substack{u<n\leq v\\n\in\mathcal A}}f(n)
    \right\rvert\geq\lambda\right)
    \leq\frac{1}{\lambda^4}
    \left(\sum_{\substack{u<n\leq v\\n\in\mathcal A}}\tau_3(n)\right)^2
\end{equation*}
by Markov's inequality and the rough hypercontractive inequality
(see \cite[Probability Result 2.3]{Harper2019_MomentsII}).
Employing Lemma \ref{lemma:billings} for the partial sums over
$\mathcal A\cap(M_l,M_{l+1}]$ and $\mathcal A\cap(N_l,N_{l+1}]$,
with threshold $\lambda/2$, and adding the two bounds gives the
desired result.
\end{proof}

\section{Proof of Theorem \ref{thm:asboundsPoly}: Polynomial images}
To prove our theorem for polynomial images, we will only need Proposition \ref{prop:approxJointGaussian} and the simpler part of Corollary \ref{cor:pairwiseNormalComparison}. Let $P$ be a polynomial as in the statement of Theorem \ref{thm:asboundsPoly} and consider\footnote{We are instead considering $\mathcal A_N\subseteq [1, c_PN^{\deg P}]$ for convenience due to the scaling of polynomial images.} $\mathcal A_N = \{P(n) : n\leq N, P(n)>0\}$, intersected with square-free numbers in the Rademacher case. Note that $\lvert \mathcal A_N\rvert \sim \kappa_PN$ for some $\kappa_P>0$ depending\footnote{In the Steinhaus case we have $\kappa_P=1$, whereas in the Rademacher case, $\kappa_P$ is the density of square-free values of $P$.} on $P$. Recall 
\begin{equation}\label{eq:S_NdefPoly}
     S_N = \frac{1}{\sqrt{\lvert \mathcal A_N\rvert}}\sum_{n\in\mathcal A_N} \mathfrak Tf(n) \sim \frac{1}{\sqrt{\kappa_P N}} \sum_{\substack{n\leq N}} \mathfrak Tf(P(n)).
\end{equation}
Moreover, for a large parameter $X$, let $N_l = \lambda^l X$ with $\lambda = \exp\left(\sqrt{\log X}\right)$, where $1\leq l\leq k = (\log X)^{\varepsilon_0}$ with $\varepsilon_0>0$ small enough to be specified later. We begin by checking that these sets satisfy properties (I), (I'), (II) and (III). 

\subsubsection*{(I)} We employ the result of \cite{ChinisShala2025}
 (\cite{KlurmanShkredovXu2023}), as well as \cite{deLaBretecheWangXu2026, Pylaiev2026}, which we recall here. 
 
 \begin{prop}
 Let $P$ be a polynomial as in the statement of Theorem \ref{thm:asboundsPoly}. There exists a constant $\delta_P>0$ such that the number of non-trivial solutions to the equation 
 \begin{equation*}
     P(n_1)P(n_2)P(n_3)P(n_4) = \square \hspace{0.5cm} (\text{or }P(n_1)P(n_2) = P(n_3)P(n_4))
 \end{equation*} with $n_i\leq N$ and $P(n_i)>0$ square-free (in the Steinhaus case, drop the square-free restriction) is $\ll N^{2-\delta_P}$. 
 \end{prop}
 
Using the above with $N = \max\{N_{l_1}, \ldots, N_{l_4}\}\ll \lambda^k X$ allows us to take $\varepsilon_1 \ll X^{-\delta_P'}$ for some slightly smaller fixed $\delta_P'>0$ (possibly depending on $P$), since $\lambda$ is so that $\lambda^k\ll_{\epsilon}X^{\epsilon}$ for all $\epsilon>0$. 

\subsubsection*{(I')} We will simply take $\varepsilon_1' =\varepsilon_1$.  

\subsubsection*{(II)} 
    We check condition (II) as follows. We introduce a smoothness parameter $y$, and split the range of $P^+(n_1) = P^+(n_2)=p$ to $p\leq y$ and $p>y$. The latter range also contains very large primes, which are dealt with separately. In any case, we show the contribution of primes $p>y$ is bounded above by 
    \begin{equation*}
\sum_{p>y}\frac{\lvert\mathcal A_{N_{l_1}}\rvert\lvert\mathcal A_{N_{l_2}}\rvert}{p^2}\ll \frac{\lvert \mathcal A_{N_{l_1}}\rvert\lvert\mathcal A_{N_{l_2}}\rvert}{y}.
    \end{equation*}
    The contribution of primes $p\leq y$ is bounded by the product of the number of $y$-smooth values of $n\in\mathcal A_N$ and the number of $y$-smooth values of $m\in\mathcal A_M$. Optimizing the choice of $y$ yields the saving $\varepsilon_2$. 
    
Concretely, using the calculations from the end of \cite[Section 4]{ChinisShala2025} (\cite[Section 2]{KlurmanShkredovXu2023}), we may take 
$$\varepsilon_2 \ll \frac{1}{y} + \frac{\psi_P(N_{l_1}, y)\psi_P(N_{l_2}, y)}{N_{l_1}N_{l_2}},$$ where $\psi_P(N, y)$ is the number of $n\leq N$ such that $P(n)$ is $y$-smooth. By a result of Hmyrova \cite{Hmyrova1966}, when $y\geq\log N$ we have the bound\footnote{The bound there is stated for irreducible $P$, but obviously the number of values of $n\leq N$ for which $P(n)$ is $y$-smooth is bounded above by the number of $n\leq N$ for which some irreducible factor of $P(n)$ is $y$-smooth.} $\psi_P(N, y)\ll N(e/u)^u$, where $u=\log N/\log y$. After an unpleasant exercise in manipulating logarithms and optimizing in $y$, 
we choose $y  = \exp\left(\sqrt{\log X\log\log X}\right)$. This gives 
$\varepsilon_2 \ll \exp(-\frac{1}{2}\sqrt{\log X\log\log X})$. 

\subsubsection*{(III)} Note that the sets $\mathcal A_N$ are nested, so (III$_a$) is satisfied. 

\subsection{Almost sure upper bound} We now have everything we need to prove the almost sure upper bound. We first apply Proposition \ref{prop:approxJointGaussian} with $k=1$ (at a single scale), yielding the following quantitative central limit theorem. 

\begin{cor}\label{cor:quantCLTpoly}
    Let $f$ be a Rademacher (or Steinhaus) random multiplicative function,
    and let $P\in\mathbb Z[x]$ be a polynomial as in Theorem
    \ref{thm:asboundsPoly}. For all
    sufficiently large $N$ and all $x\in\mathbb R$, we have
    \begin{multline*}
        \left\lvert\mathbb P\left(\frac{1}{\sqrt{\kappa_P N}}
        \sum_{\substack{n\leq N\\P(n)>0}}\mathfrak T f(P(n))\leq x\right)
        -\frac{1}{\sqrt{2\pi}}\int_{-\infty}^x
        e^{-\frac{t^2}{2}}\,\mathrm{d}t\right\rvert\\
        \ll_P \frac{1}{1+\lvert x\rvert^{\frac{16}{5}}}
        \exp\left(-\frac{1}{10}\sqrt{\log N\log\log N}\right).
    \end{multline*}
\end{cor}

We will use Corollary \ref{cor:quantCLTpoly} at certain test points $(N_l)_{l=1}^\infty$, combined with the following almost sure slow variation result between the test points.

\begin{lemma}\label{lemma:slowVarPoly} Let $f$ be a Rademacher (or Steinhaus) random multiplicative function, and let $P\in\mathbb Z[x]$ be a polynomial as in Theorem \ref{thm:asboundsPoly}. 
For any constant $A>0$, there exists a constant $c>0$ such that for $N_l = \lfloor e^{l^c}\rfloor$, we have that
    \begin{equation*}
        \max_{N_{l}<N\leq N_{l+1}} \left\lvert\sum_{n\leq N} f(P(n)) - \sum_{n\leq N_{l}} f(P(n)) \right\rvert \ll \frac{\sqrt{N_{l+1}}}{(\log N_{l+1})^A}
    \end{equation*} holds almost surely for every large enough $l\in\mathbb N$. 
\end{lemma}

\begin{proof}
    By Proposition \ref{prop:slowVar}, we have 
\begin{equation*}
\mathbb P\left(\max_{N_l<N\leq N_{l+1}}\left\lvert\sum_{N_l<n\leq N}f(P(n))\right\rvert\geq x\right) \ll \frac{1}{x^4} \left(\sum_{\substack{N_l<n\leq N_{l+1} }} \tau_3(P(n))\right)^2 
\end{equation*}
Lemma \ref{lemma:PolyBound} gives 
\begin{equation*}
        \sum_{N_l<n\leq N_{l+1}} \tau_{3}(P(n)) \ll (N_{l+1} - N_l)(\log N_{l+1})^{c_P},
    \end{equation*} 
    for some constant $c_P>0$ depending on $P$. Taking $x = \sqrt{N_{l+1}}/(\log N_{l+1})^A$, we obtain 
    \begin{equation*}
        \mathbb P\left(\max_{N_l<N\leq N_{l+1}}\left\lvert\sum_{N_l<n\leq N}f(P(n))\right\rvert\geq x\right) \ll \left(\frac{N_{l+1} - N_l}{N_{l+1}} \right)^2 (\log N_{l+1})^{2c_P +4A}.
    \end{equation*}
Plugging in $N_l = \lfloor e^{l^c}\rfloor$ and summing over $l$ to get a convergent series, we see that choosing $0<c<1/({2c_P+4A+2})$ suffices. Applying Borel-Cantelli finishes the proof.
\end{proof} 

We finally prove the almost sure upper bound. By Lemma \ref{lemma:slowVarPoly} with $A=1$, it suffices to show that almost surely there exists a constant $C>0$ such that for all $l\in\mathbb N$ we have
    \begin{equation}\label{eq:ubPoly}
        \left\lvert\sum_{n\leq N_l} f(P(n))\right\rvert\leq C \sqrt{N_l\log\log N_l},
    \end{equation}
    where $N_l$ is as in the statement of the lemma (in the Steinhaus case, we do this separately for the real and imaginary parts). By Corollary \ref{cor:quantCLTpoly}, the probability that the above fails for a fixed $l$ and large enough implied constant $C$ is $\ll (\log N_l)^{-\frac{C^2}{2}}$. Making $C$ larger (in terms of $c$) if necessary, we have that the series of probabilities $$\sum_{l=1}^\infty (\log N_l)^{-\frac{C^2}{2}} = \sum_{l=1}^\infty \frac{1}{l^{\frac{cC^2}{2}}}<\infty,$$ thus \eqref{eq:ubPoly} holds almost surely by Borel-Cantelli (with a possibly even larger constant $C$ to account for the finitely many failures).

\subsection{Almost sure lower bound} Here we will use Corollary \ref{cor:pairwiseNormalComparison} with the scales $N_1, N_2, \ldots, N_k$ defined at the start of the section. In fact, we will prove something slightly stronger.

\begin{theorem}\label{thm:localizedBoundPoly}
    Let $f$ be a Rademacher (or Steinhaus) random multiplicative function and let $P\in\mathbb Z[x]$ be a polynomial as in Theorem \ref{thm:asboundsPoly}. There exists a constant $c>0$ such that for any large enough $X$, we have that
    \begin{equation*}
        \max_{X<N\leq X^2} \frac{1}{\sqrt N} \left\lvert\sum_{n\leq N}f(P(n)) \right\rvert \gg \sqrt{\log\log X}
    \end{equation*}
    holds with a suitably small implied constant with probability $1-O((\log X)^{-c})$.
\end{theorem}

In light of Lemma \ref{lemma:asLbGeneral}, this suffices to establish the almost sure lower bound in Theorem \ref{thm:asboundsPoly}. Henceforth we prove Theorem \ref{thm:localizedBoundPoly}. Applying Corollary \ref{cor:pairwiseNormalComparison} we obtain the following bound.

\begin{cor}\label{cor:MCLTpoly}
    Let $t\geq 0$ and let $p_t$ be as in Corollary \ref{cor:pairwiseNormalComparison}. Then
    \begin{equation*}
        \mathbb P\left(\max_{1\leq l\leq k} S_{N_l} \leq t\right)
        \ll \frac{1}{kp_t} + \frac{1+t}{p_t^2}\exp\left(-\frac{1}{15}\sqrt{\log X}\right).
    \end{equation*}
\end{cor}

\begin{proof}
    Recalling that $\lvert\mathcal A_{N_l}\rvert\asymp N_l$ and $\lambda=\exp(\sqrt{\log X})$, we have
    \begin{equation*}
        \max_{1\leq l_1<l_2\leq k}\frac{\lvert\mathcal A_{N_{l_1}}\rvert}{\lvert\mathcal A_{N_{l_2}}\rvert}
        \ll \frac{1}{\lambda} = \exp\left(-\sqrt{\log X}\right).
    \end{equation*}
    Since $\varepsilon_1+\varepsilon_1'$ is bounded by a negative power of $X$ and $\varepsilon_2\ll\exp\left(-\frac{1}{2}\sqrt{\log X\log\log X}\right)$, the result follows from Corollary \ref{cor:pairwiseNormalComparison}.
\end{proof}

Taking $t=\sqrt{\log k}$, we have $p_t\asymp k^{-1/2}/\sqrt{\log k}$ by the Gaussian tail estimate. Thus, recalling that $k=\lfloor(\log X)^{\varepsilon_0}\rfloor$, Corollary \ref{cor:MCLTpoly} gives
\begin{equation*}
    \mathbb P\left(\max_{1\leq l\leq k} S_{N_l}\leq\sqrt{\log k}\right)
    \ll \frac{\sqrt{\log k}}{\sqrt k}
    \ll (\log X)^{-\varepsilon_0/3}.
\end{equation*}
Now $\sqrt{\log k}\asymp\sqrt{\log\log X}$, and for $0<\varepsilon_0<1/2$ we have
\begin{equation*}
    X<N_1\leq N_l\leq N_k\leq X\exp\left((\log X)^{\frac{1}{2}+\varepsilon_0}\right)\leq X^2
\end{equation*}
for all $1\leq l\leq k$ and large enough $X$. Recalling \eqref{eq:S_NdefPoly} and that $\lvert\mathcal A_{N_l}\rvert\asymp N_l$, Theorem \ref{thm:localizedBoundPoly} follows with $c=\varepsilon_0/3$.

\section{Proof of Theorem \ref{thm:asBoundsShort}: Short intervals} Here we will require the more sophisticated part of Corollary \ref{cor:pairwiseNormalComparison}. Let $H$ be as in the statement of Theorem \ref{thm:asBoundsShort} and consider $\mathcal A_N = \{n\in\mathbb N : N-H<n\leq N\}$, intersected with the square-free integers in the Rademacher case. Note that $\lvert\mathcal A_N\rvert\sim \kappa H$ for some\footnote{In the Steinhaus case we have $\kappa = 1$, whereas in the Rademacher case $\kappa = \frac{6}{\pi^2}$.} $\kappa>0$. Recall 
\begin{equation*}
    S_N = \frac{1}{\sqrt{\lvert\mathcal A_N\rvert}}\sum_{n\in\mathcal A_N} \mathfrak Tf(n)\sim \frac{1}{\sqrt{\kappa H}}\sum_{N-H<n\leq N}\mathfrak Tf(n).
\end{equation*}
Moreover, let $X$ be a large parameter. For some small $\delta>0$ depending on the function $H$ and then even smaller $\varepsilon_0>0$, find $k= (X/H(X))^{\varepsilon_0}$ primes $q_1< q_2< \cdots< q_k$ in the range $h(X)/2<q_l\leq h(X)$, where $h(X) := (X/H(X))^{\delta}$. We will not fix $\delta$ or $\varepsilon_0$ for now, as there will be various (but finitely many) places where we may have to make them smaller. For these values of $q_l$, let $N_l = q_lX$ with corresponding $H_l = H(N_l)$. 
Observe that for large enough $X$ and small enough $\delta>0$, the sets $\mathcal A_{N_l}$ are disjoint by the assumptions on the function $H$ of Theorem \ref{thm:asBoundsShort}. It is not these sets that we will apply Proposition \ref{prop: complicatedapproxJointG} to. If we did this, the very smooth integers in the sets $\mathcal A_{N_l}$ would prevent us from obtaining a suitably small value for $\varepsilon_2$. We bypass this by splitting $S_{N_l}$ as 
\begin{equation}\label{eq:splitting1}
    \sum_{N_l - H_l<n\leq N_l}f(n) = \sum_{(Xh(X))^{\frac{2}{3}}< p}f(p)\sum_{\substack{N_l-H_l < n\leq N_l \\ p\mid n}}f\left(\frac np\right) + \sum_{\substack{N_l-H_l<n\leq N_l \\ P^+(n)\leq (Xh(X))^{\frac{2}{3}}}}f(n).
\end{equation}
If $H_1> N_1/(\log N_1)^2$ (hence since $H$ is increasing we have $H_l \gg Xh(X)/(\log X)^2$ for all values of $l$), we choose some $\varepsilon>0$ to be specified later and additionally split the first part as 
\begin{equation}\label{eq:splitting2}
    \sum_{N_l - H_l<n\leq N_l}f(n) = \sum_{(Xh(X))^{\frac{2}{3}}< p }f(p)\sum_{\substack{N_l-H_l < n\leq N_l \\ p\mid n \\ \Omega(n)\leq (1+\varepsilon)\log\log X}}f\left(\frac np\right) \\ + \sum_{\substack{N_l-H_l<n\leq N_l \\  P^+(n)\leq (Xh(X))^{\frac{2}{3}} \\ \text{or } \Omega(n)> (1+\varepsilon)\log\log X }}f(n).
\end{equation}
We first show that the sum over integers with $\Omega(n)> (1+\varepsilon)\log\log X$ above may be typically ignored for all values of $l$. 

\begin{lemma}\label{lemma:ignoreTooManyPrimesShort}
Assume that $H_1 > N_1/(\log N_1)^2$ and let $\varepsilon>0$. There exists $\varepsilon' = \varepsilon'(\varepsilon)>0$ such that with probability $1-O((\log X)^{-\varepsilon'+2\varepsilon_0})$ we have 
\begin{equation*}
    \frac{1}{\sqrt{H_l}}\sum_{\substack{N_l-H_l<n\leq N_l \\ \Omega(n)>(1+\varepsilon)\log\log X}} f(n)\ll1
\end{equation*} for all $1\leq l\leq k.$\end{lemma}

\begin{proof}
    The probability that $$\left\lvert\sum_{\substack{N_l - H_l < n\leq N_l \\ \Omega(n)>(1+\varepsilon)\log\log X}} f(n)\right\rvert\geq\lambda$$ is bounded by $\lambda^{-2}$ times the number of $n\in (N_l - H_l, N_l]$ with $\Omega(n)>(1+\varepsilon)\log\log X$ by Markov's inequality using a second moment estimate. The number of such $n$ is $\ll H_l/(\log X)^{\varepsilon'}$ by Lemma \ref{lemma:tooManyPrimes}. Taking a union bound over $k$ values of  $l$ and recalling that $k=(X/H(X))^{\varepsilon_0}\ll (\log X)^{2\varepsilon_0}$ (making $\varepsilon_0>0$ smaller if necessary) gives a probability $1-O((\log X)^{-\varepsilon'+2\varepsilon_0})$ that all of the sums are bounded. 
\end{proof}

If $H_1\leq N_1/(\log N_1)^2$ (in which case by concavity of the function $H$ we have $H_l$  \linebreak$\ll Xh(X)/(\log X)^2$ for all values of $l$), let $\mathcal B_{N_l}$ be the subset of $\mathcal A_{N_l}$ of those integers $n$ with $P^+(n)>(Xh(X))^{\frac{2}{3}}$. If $H_1> N_1/(\log N_1)^2$, add the restriction $\Omega(n)\leq(1+\varepsilon)\log\log X$. Now we work with 
\begin{equation*}
    \sum_{n\in\mathcal B_{N_l}}f(n) + O\left(\left\lvert\sum_{\substack{N_l-H_l<n\leq N_l \\  P^+(n)\leq (Xh(X))^{\frac{2}{3}}}}f(n)\right\rvert\right).
\end{equation*}
We further show that for many values of $l$, the smooth sum is typically small. 

\begin{lemma}\label{lemma:ignoresmoothShort}
    For large enough $X$, with probability $1-O((\log \log X)^{-\frac{2}{3}})$ over $f(p)$ with $p\leq (Xh(X))^{\frac{2}{3}}$, there exist $\geq \frac{k}{2}$ values of $l$ such that 
    \begin{equation*}
        \frac{1}{\sqrt H_l}\sum_{\substack{N_l-H_l<n\leq N_l \\ P^+(n)\leq (Xh(X))^{\frac{2}{3}}}}f(n)\ll (\log\log X)^{\frac{1}{3}}.    \end{equation*}
\end{lemma} 

\begin{proof}
 The probability that the normalized smooth sum is greater than $\lambda$ in size is bounded by $1/\lambda^2$ by Markov's inequality using a second moment estimate. We choose  $\lambda = (\log\log X)^{1/3}$ so that the expected number of $l$ with the normalized smooth sum greater than $\lambda$ is $\leq k/\lambda^2$. Then for large enough $X$, with probability $1-O((\log \log X)^{-\frac{2}{3}})$ we may find $\geq \frac{k}{2}$ suitable values of $l$. 
\end{proof}

Henceforth condition on the values of $f(p)$ for $p\leq (Xh(X))^{\frac{2}{3}}$, upon which we may find $\geq \frac{k}{2}$ fixed values of $l$ that we will refer to as \emph{good}, with the property that
\begin{equation*}
    \sum_{N_l - H_l < n\leq N_l}f(n) = \sum_{n\in\mathcal B_{N_l}}f(n) + \sum_{\substack{N_l-H_l<n\leq N_l \\  P^+(n)>(Xh(X))^{\frac{2}{3}} \\ \Omega(n)>(1+\varepsilon)\log\log X}} f(n) + O\left(\sqrt{H_l}(\log \log X)^{\frac{1}{3}} \right)
\end{equation*} for all good $l$. (The second sum appears only when $H_1 > N_1/(\log N_1)^2$.)

\subsection{Almost sure lower bound}
Throughout this section, all values of $l$ will be good even if we do not state so. As before, we will actually prove something slightly stronger.
\begin{theorem}\label{thm:localizedBoundShort}
    Let $f$ be a Rademacher (or Steinhaus) random multiplicative function and let $H$ be a function as in Theorem \ref{thm:asBoundsShort}. For any large enough $X$, we have that 
    \begin{equation*}
        \max_{X<N\leq X^2} \frac{1}{\sqrt H} \left\lvert\sum_{N-H<n\leq N}f(n) \right\rvert \gg \sqrt{\log \frac{X}{H(X)}}
    \end{equation*} holds with a suitably small implied constant with probability $1-O((\log \log X)^{-\frac{2}{3}})$. 
\end{theorem}

In light of Lemma \ref{lemma:asLbGeneral}, this suffices to establish Theorem \ref{thm:asBoundsShort}. Henceforth we prove Theorem \ref{thm:localizedBoundShort}. 

Let
\begin{equation*}
    \widetilde S_{N_l} = \frac{1}{\sqrt{\lvert \mathcal B_{N_l}\rvert}}\sum_{n\in\mathcal B_{N_l}} \mathfrak Tf(n).
\end{equation*}
We wish to apply Corollary \ref{cor:pairwiseNormalComparison} to the $\mathcal B_{N_l}$, so we will check conditions (I'), (II), (III), and (IV) for these sets.

\subsubsection*{(I')}  If $n_i\in\mathcal B_{N_{l_i}}$ are such that $n_1n_2n_3n_4 = \square$ (or $n_1n_4 = n_2n_3$) with $P^+(n_1) = \cdots = P^+(n_4)=p > (Xh(X))^{\frac{2}{3}}$, then $n_i' = n_i/p$ satisfy $n_1'n_2'n_
  3'n_4'=\square$ (or $n_1'n_4' = n_2'n_3'$) and $n_i'\leq N_{l_i}/p\leq Xh(X)/p$. For a fixed $p$, the number of such $n_i'$ is $\ll (Xh(X)\log X/p)^2$. Over all primes $p>(Xh(X))^{\frac{2}{3}}$, we have at most 
\begin{equation*}
    \ll (Xh(X)\log X)^2 \sum_{(Xh(X))^{\frac{2}{3}}<p\leq Xh(X)}\frac{1}{p^2}\ll (Xh(X))^{\frac{4}{3}}\log X
\end{equation*}
solutions for $n_i$. For  $\varepsilon_1' = X^{-(\frac{2}{3} - \frac{8}{15})}\log X$, this is $\ll \varepsilon_1' H_1^2\ll \varepsilon_1'\sqrt{H_{l_1}\cdots H_{l_4}}$ uniformly in the values of the $l_i$, since $H\geq N^{\frac{11}{15}}.$

\subsubsection*{(II)} The number of $n_1\in\mathcal B_{N_{l_1}}, n_2\in\mathcal B_{N_{l_2}}$ with $P^+(n_1)=P^+(n_2)$ is bounded by 
\begin{multline*}
    \sum_{(Xh(X))^{\frac{2}{3}}<p} \#\{n_1\in\mathcal B_{N_{l_1}} : P^+(n_1) = p \} \cdot \#\{n_2\in\mathcal B_{N_{l_2}} : P^+(n_2) = p \}  \\ \ll \left(\frac{H_{l_1}}{(Xh(X))^{\frac{2}{3}}} + 1\right)\sum_{(Xh(X))^{\frac{2}{3}}<p} \#\{n_2\in\mathcal B_{N_{l_2}} : P^+(n_2) = p \} \\  \ll \left(\frac{H_{l_1}}{(Xh(X))^{\frac{2}{3}}} + 1 \right)H_{l_2}.
\end{multline*}
For $\varepsilon_2 = (Xh(X))^{-\frac{2}{3}}$, this is certainly $\ll \varepsilon_2 H_{l_1}H_{l_2}$ uniformly in the values of $l_1$ and $l_2$. 

\subsubsection*{(III)} As we remarked in the beginning, the stronger assumption (III$_b$) is satisfied here. 

\subsubsection*{(IV)} We have chosen $P = (Xh(X))^{\frac{2}{3}}$ and the rest is true by definition of the $\mathcal B_{N_l}$. 

Now we apply Corollary \ref{cor:pairwiseNormalComparison} to the $\widetilde{S}_{N_l}$ with good $l$. By the fourth moment estimate in the proof of Proposition \ref{prop: complicatedapproxJointG} with $k=1$, Markov's inequality and a union bound give that, with probability $1-O(k\sqrt{\varepsilon_1'+\varepsilon_2})$ over the $f(p)$ for primes $p\leq P$, the sum of the conditional fourth moments of the summands indexed by $p>P$ is $\ll\sqrt{\varepsilon_1'+\varepsilon_2}$ for each of the $k$ chosen indices. The same bound then holds for every unit linear combination of any two of the $\widetilde{S}_{N_l}$, since $\lvert a+b\rvert^4\leq 8(\lvert a\rvert^4+\lvert b\rvert^4)$.

Suppose also that $1/2\leq\widetilde{\mathbb E}\widetilde{S}_{N_l}^2\leq 2$ for all good $l$. If the absolute conditional correlations are at most $1/2$, then the conditional variances of all the above linear combinations of good coordinates are $\asymp 1$, so the conditional bound in dimension two holds simultaneously for every good pair. If any absolute conditional correlation exceeds $1/2$, the bound below is trivial. Thus, recalling that there are at least $k/2$ good indices, Corollary \ref{cor:pairwiseNormalComparison} gives, for $t\geq 0$ and $p_t$ as in that corollary,
\begin{equation}\label{eq:shortIntsNormal}
    \widetilde{\mathbb P}\left(\max_l
    \frac{\widetilde{S}_{N_l}}{\sqrt{\widetilde{\mathbb E}\widetilde{S}_{N_l}^2}}
    \leq t\right)
    \ll \frac{1}{kp_t}
    + (1+t)^2\max_{l_1<l_2}
    \frac{\left\lvert\widetilde{\mathbb E}
    \widetilde{S}_{N_{l_1}}\widetilde{S}_{N_{l_2}}\right\rvert}
    {\sqrt{\widetilde{\mathbb E}\widetilde{S}_{N_{l_1}}^2\,
    \widetilde{\mathbb E}\widetilde{S}_{N_{l_2}}^2}}
  + \frac{1+t}{p_t^2}\left(\varepsilon_1'+\varepsilon_2\right)^{\frac{1}{30}},
\end{equation}
where both maxima are over good indices. Since each $n\in\mathcal B_{N_l}$ has a unique prime factor $p>P$ to multiplicity one, averaging over the $f(p)$ for $p>P$ gives
\begin{equation*}
    \widetilde{\mathbb E}\widetilde{S}_{N_{l_1}}\widetilde{S}_{N_{l_2}}
    = \frac{1}{\sqrt{\lvert\mathcal B_{N_{l_1}}\rvert
    \lvert\mathcal B_{N_{l_2}}\rvert}}
    \sum_{p>P}\sum_{\substack{n_1\in\mathcal B_{N_{l_1}} \\
    n_2\in\mathcal B_{N_{l_2}} \\ P^+(n_1)=P^+(n_2)=p}}
    \Re\left(f(n_1/p)\overline{f(n_2/p)}\right).
\end{equation*}
It remains to check the conditional variance bounds and to bound the absolute conditional correlations uniformly for distinct good indices. The heart of the matter lies in the following proposition.

\begin{prop}\label{prop:CountAt4Scales}
    Assume that $H$ is as in Theorem \ref{thm:asBoundsShort}. Let $q_{l_1}$ and $q_{l_2}$ be primes in the range $h(X)/2<q_{l_1},q_{l_2}\leq h(X)=(X/H(X))^{\delta}$. The number of non-trivial quadruples $n_1,n_3\in\mathcal B_{N_{l_1}}$ and $n_2,n_4\in\mathcal B_{N_{l_2}}$ with $P^+(n_1)=P^+(n_2)$ and $P^+(n_3)=P^+(n_4)$ such that $n_1n_2n_3n_4=\square$ (or $n_1n_4=n_2n_3$), is $\ll H_{l_1}H_{l_2}/h(X)^2$ if $\delta>0$ and $\varepsilon>0$ are small enough.\footnote{Recall the definition of $\mathcal B_{N_l}$ for the dependence on $\varepsilon$. In the Rademacher case the $n_i$ are restricted to square-free integers.}
\end{prop}

\begin{proof}
Write $\gamma_l=H_l/N_l$ so that $H_l=\gamma_lN_l$. Since $H$ is increasing and concave and $N_l=q_lX$, we have
\begin{equation*}
    \gamma_{l_1}\asymp\gamma_{l_2}\asymp\gamma=\gamma(X)
    =\frac{H(Xh(X))}{Xh(X)}\ll\frac{H(X)}{X}
\end{equation*}
uniformly in the values of $l_1,l_2$. We first suppose that $q_{l_1}\neq q_{l_2}$.

We begin with the Steinhaus case. We utilize the parametrization of the solutions to $n_1n_4=n_2n_3$ as
\begin{equation*}
    n_1=ga,\quad n_2=gb,\quad n_3=ha,\quad n_4=hb,
\end{equation*}
where $(a,b)=1$ and $g=(n_1,n_2)$, $h=(n_3,n_4)$. In particular, $P^+(n_1)\mid g$ and $P^+(n_3)\mid h$, so both $g,h>(Xh(X))^{\frac{2}{3}}$ and $a,b\ll N_{l_1}^{\frac{1}{3}}$. Note that
\begin{equation*}
    1-\gamma_{l_1}\leq\frac{q_{l_2}a}{q_{l_1}b}
    \leq\frac{1}{1-\gamma_{l_2}},
\end{equation*}
thus
\begin{equation}\label{eq:a,b,gamma}
    \left\lvert\frac{q_{l_2}a}{q_{l_1}b}-1\right\rvert\ll\gamma.
\end{equation}
If $q_{l_2}a=q_{l_1}b$, then $a=q_{l_1}$ and $b=q_{l_2}$ by coprimality. Using the interval for $n_1$ to count $g$ and the interval for $n_4$ to count $h$, the number of choices is
\begin{equation*}
    \ll\frac{H_{l_1}H_{l_2}}{q_{l_1}q_{l_2}}
    \ll\frac{H_{l_1}H_{l_2}}{h(X)^2}.
\end{equation*}
Otherwise $1\leq\lvert q_{l_2}a-q_{l_1}b\rvert\ll\gamma q_{l_1}b$, so both $q_{l_2}a,q_{l_1}b\gg\gamma^{-1}$.

Let us first handle the case when $H_1\leq N_1/(\log N_1)^2$. Given $a,b$, the number of choices for $g,h$ is $\ll H_{l_1}H_{l_2}/(ab)$, again using the intervals for $n_1,n_4$. Here $H_{l_1}/a,H_{l_2}/b\gg 1$ by the lower bound on $H$. By \eqref{eq:a,b,gamma}, we have $b\asymp a$, and for fixed $a$ there are $\ll 1+\gamma a$ possible values of $b$. Therefore
\begin{align*}
    \sum_{\substack{a,b\ll N_{l_1}^{\frac{1}{3}} \\
    q_{l_2}a,q_{l_1}b\gg\gamma^{-1} \\ \eqref{eq:a,b,gamma}}}\frac{1}{ab}
    &\ll \sum_{\substack{a\ll N_{l_1}^{\frac{1}{3}} \\
    a\gg(q_{l_2}\gamma)^{-1}}}\frac{1+\gamma a}{a^2}
    \ll \gamma(\log X+q_{l_2}).
\end{align*}
In this case $\gamma\ll(\log X)^{-2}$. Since $h(X)\ll\gamma^{-\delta}$, we have
\begin{equation*}
    \gamma h(X)^2\log X\ll(\log X)^{-1+4\delta}
    \qquad\text{and}\qquad
    \gamma h(X)^3\ll\left(\frac{X}{H(X)}\right)^{-1+3\delta}.
\end{equation*}
Thus $\gamma(\log X+q_{l_2})\ll h(X)^{-2}$ for $\delta<1/4$, giving the desired bound.

Let us now turn to the case when $H_1>N_1/(\log N_1)^2$. Let $K=(1+\varepsilon)\log\log X$. Since $\Omega(ghab)\leq 2K$, the number of solutions with $q_{l_2}a\neq q_{l_1}b$ is
\begin{equation}\label{eq:SteinhausLambda}
    \ll 2^{2K}
    \sum_{\substack{a,b\ll N_{l_1}^{\frac{1}{3}} \\
    q_{l_2}a,q_{l_1}b\gg\gamma^{-1} \\ \eqref{eq:a,b,gamma}}}
    2^{-\Omega(ab)}
    \sum_{\substack{(N_{l_1}-H_{l_1})/a<g\leq N_{l_1}/a \\
    (N_{l_2}-H_{l_2})/b<h\leq N_{l_2}/b}}
    2^{-\Omega(gh)}.
\end{equation}
Here we have retained only the restrictions on $g,h$ coming from $n_1,n_4$. Applying Proposition \ref{prop:nair-tenenbaum} with the function $2^{-\Omega(n)}$ to each of these intervals, the right-hand side of \eqref{eq:SteinhausLambda} is
\begin{equation}\label{eq:SteinhausLambda2}
    \ll 2^{2K}\frac{H_{l_1}H_{l_2}}{\log X}
    \sum_{\substack{a\ll N_{l_1}^{\frac{1}{3}}\\
    q_{l_2}a\gg\gamma^{-1}}}\frac{2^{-\Omega(a)}}{a^2}
    \sum_{b:\,\eqref{eq:a,b,gamma}}2^{-\Omega(b)}.
\end{equation}
Indeed, the upper endpoints are $\gg N_{l_1}^{2/3}$ and the lengths are $\gg 1/(\log X)^2$ times the respective upper endpoints.

For $a\leq\gamma^{-2}$, we simply bound the sum over $b$ by $O(1+\gamma a)$. The contribution to the double sum in \eqref{eq:SteinhausLambda2} is therefore
\begin{equation*}
    \ll\sum_{(q_{l_2}\gamma)^{-1}\ll a\leq\gamma^{-2}}
    \frac{1+\gamma a}{a^2}
    \ll\gamma\left(q_{l_2}+\log(2/\gamma)\right).
\end{equation*}
For $a>\gamma^{-2}$, the possible values of $b$ lie in an interval at height $\asymp a$ which we may take to have length $\asymp\gamma a\gg\sqrt a$. Proposition \ref{prop:nair-tenenbaum} bounds the sum over $b$ by $\ll\gamma a/(\log(2a))^{1/2}$. By a dyadic decomposition and another application of Proposition \ref{prop:nair-tenenbaum}, we have
\begin{equation*}
    \sum_{a\ll N_{l_1}^{\frac{1}{3}}}
    \frac{2^{-\Omega(a)}}{a(\log(2a))^{\frac{1}{2}}}
    \ll 1+\sum_{\substack{A=2^k\ll X^2\\ k\geq 1}}
    \frac{1}{A(\log A)^{\frac{1}{2}}}
    \sum_{A<a\leq 2A}2^{-\Omega(a)}
    \ll 1+\sum_{1\leq k\ll\log X}\frac{1}{k}
    \ll\log\log X.
\end{equation*}
Since $\gamma\gg(\log X)^{-2}$ in this case, the double sum in \eqref{eq:SteinhausLambda2} is $\ll\gamma(q_{l_2}+\log\log X)$. Accounting for the solutions with $q_{l_2}a=q_{l_1}b$, the number of non-trivial solutions is therefore
\begin{equation}\label{eq:SteinhausLambda3}
    \ll 2^{2K}\frac{H_{l_1}H_{l_2}}{\log X}
    \gamma(q_{l_2}+\log\log X)
    +\frac{H_{l_1}H_{l_2}}{h(X)^2}
    \ll\frac{H_{l_1}H_{l_2}}{h(X)^2}.
\end{equation}
For the last bound, recall that $H(X)\leq X/(\log X)^c$ with $c>2\log 2-1$. Taking $\varepsilon,\delta>0$ small enough that $2(1+\varepsilon)\log 2-1<c(1-3\delta)$ gives
\begin{equation*}
    \frac{2^{2K}}{\log X}\gamma h(X)^3\log\log X
    \ll (\log X)^{2(1+\varepsilon)\log 2-1-c(1-3\delta)}
    \log\log X\ll 1.
\end{equation*}

Next, we consider the Rademacher case. Write $A=(n_1,n_2)$ and $B=(n_3,n_4)$. The integers $(n_1/A)(n_2/A)$ and $(n_3/B)(n_4/B)$ are square-free and their product is a square, so they are equal. Factoring this common square-free integer according to its prime divisors gives the parametrization
\begin{equation*}
    n_1=Aru,\quad n_2=Asv,\quad n_3=Brv,\quad n_4=Bsu,
\end{equation*}
where $r,s,u,v$ are pairwise coprime. As before, $A,B>(Xh(X))^{\frac{2}{3}}$, and hence $ru,rv,su,sv\ll N_{l_2}^{\frac{1}{3}}$. Since
\begin{equation*}
    \frac{q_{l_2}^2n_1n_3}{q_{l_1}^2n_2n_4}
    =\frac{(q_{l_2}r)^2}{(q_{l_1}s)^2}
    \qquad\text{and}\qquad
    \frac{n_1n_4}{n_2n_3}=\frac{u^2}{v^2},
\end{equation*}
the interval restrictions give
\begin{equation}\label{eq:r,s,gamma}
    \left\lvert\frac{q_{l_2}r}{q_{l_1}s}-1\right\rvert\ll\gamma
\end{equation}
and
\begin{equation}\label{eq:u,v,gamma}
    \left\lvert\frac{u}{v}-1\right\rvert\ll\gamma.
\end{equation}
If $u=v$, then $u=v=1$ by coprimality and $n_1n_4=n_2n_3$, so the Steinhaus count already gives the desired bound for these solutions. Henceforth assume that $u\neq v$, so that $u,v\gg\gamma^{-1}$ by integrality. Similarly, either $q_{l_2}r,q_{l_1}s\gg\gamma^{-1}$ or $q_{l_2}r=q_{l_1}s$. In the latter case $r=q_{l_1}$ and $s=q_{l_2}$.

For fixed $r,s,u,v$, we retain the restrictions on $A,B$ coming from $n_1,n_4$, namely
\begin{equation*}
    \frac{N_{l_1}-H_{l_1}}{ru}<A\leq\frac{N_{l_1}}{ru}
    \qquad\text{and}\qquad
    \frac{N_{l_2}-H_{l_2}}{su}<B\leq\frac{N_{l_2}}{su}.
\end{equation*}
First suppose that $H_1\leq N_1/(\log N_1)^2$. The number of choices for $A,B$ is $\ll H_{l_1}H_{l_2}/(rsu^2)$. Dropping the restrictions between the pairs $(r,s)$ and $(u,v)$, we may sum $1/(rs)$ and $1/u^2$ separately, keeping $r,s,u,v\ll N_{l_2}^{1/3}$. For $q_{l_2}r\neq q_{l_1}s$, the sum over $r,s$ is $\ll\gamma(\log X+q_{l_2})$ by the unweighted estimate following \eqref{eq:a,b,gamma}. The case $q_{l_2}r=q_{l_1}s$ adds only $1/(q_{l_1}q_{l_2})$. For the sum over $u,v$, the same estimate with $q_{l_1},q_{l_2}$ replaced by $1$ gives
\begin{equation*}
    \sum_{\substack{\gamma^{-1}\ll u\ll N_{l_2}^{\frac{1}{3}}}}
    \frac{1}{u^2}\sum_{v:\,\eqref{eq:u,v,gamma}}1
    \ll\gamma\log X.
\end{equation*}
Thus the number of solutions with $u\neq v$ is
\begin{equation}\label{eq:RademacherLambda1}
    \ll H_{l_1}H_{l_2}
    \left(\frac{1}{q_{l_1}q_{l_2}}+\gamma(\log X+q_{l_2})\right)
    \gamma\log X
    \ll\frac{H_{l_1}H_{l_2}}{h(X)^2},
\end{equation}
using the bounds from the Steinhaus case.

Finally, suppose that $H_1>N_1/(\log N_1)^2$. In this case $\Omega(ABrsuv)\leq 2K$, since $(ABrsuv)^2=n_1n_2n_3n_4$. Inserting the weight $2^{2K-\Omega(ABrsuv)}$ and applying Proposition \ref{prop:nair-tenenbaum} to the displayed intervals for $A,B$, the number of solutions with $u\neq v$ is
\begin{equation}\label{eq:RademacherLambda2}
    \ll 2^{2K}\frac{H_{l_1}H_{l_2}}{\log X}
    \sum_{\substack{ru,rv,su,sv\ll N_{l_2}^{\frac{1}{3}}\\
    (r,s)=1,\ u\neq v\\ \eqref{eq:r,s,gamma},\ \eqref{eq:u,v,gamma}}}
    \frac{2^{-\Omega(rsuv)}}{rsu^2}.
\end{equation}
Again we drop the restrictions between $(r,s)$ and $(u,v)$ and use $s\asymp r$. For $q_{l_2}r\neq q_{l_1}s$, the weighted sum over $r,s$ is $\ll\gamma(q_{l_2}+\log\log X)$ by the estimate for the double sum in \eqref{eq:SteinhausLambda2}. The case $q_{l_2}r=q_{l_1}s$ contributes at most $1/(q_{l_1}q_{l_2})$. The weighted sum over $u,v$ is $\ll\gamma\log\log X$ by the same estimate with $q_{l_1},q_{l_2}$ replaced by $1$, since $u\neq v$. It follows that the right-hand side of \eqref{eq:RademacherLambda2} is
\begin{equation}\label{eq:RademacherLambda3}
    \ll 2^{2K}\frac{H_{l_1}H_{l_2}}{\log X}
    \left(\frac{1}{q_{l_1}q_{l_2}}+\gamma(q_{l_2}+\log\log X)\right)
    \gamma\log\log X
    \ll\frac{H_{l_1}H_{l_2}}{h(X)^2},
\end{equation}
as before.

It remains to consider $l_1=l_2$. In the Steinhaus case, $a=b$ gives only trivial solutions. For $a\neq b$, \eqref{eq:a,b,gamma} now gives $a,b\gg\gamma^{-1}$, and the same estimates apply with $q_{l_1},q_{l_2}$ replaced by $1$ in the sums over $a,b$. In the Rademacher case, $u=v=1$ is treated as before. If $r=s$, then $r=s=1$ and $n_1n_3=n_2n_4$; interchanging $n_3,n_4$ reduces this to the Steinhaus count, since they now belong to the same set. Otherwise $r\neq s$ and $u\neq v$, and the preceding estimates again apply with $q_{l_1},q_{l_2}$ replaced by $1$ in the sums over $r,s$. This finishes the proof.
\end{proof}

We now estimate the conditional covariances. For $l_1\neq l_2$ the sets $\mathcal B_{N_{l_1}}$ and $\mathcal B_{N_{l_2}}$ are disjoint, whereas for $l_1=l_2$ the terms with $n_1=n_2$ contribute exactly $1$. After subtracting these terms in the latter case, the non-trivial solutions in the second moment are bounded by Proposition \ref{prop:CountAt4Scales}, including when $l_1=l_2$. The remaining trivial solutions are crossed pairings with a common largest prime factor, and contribute $\ll\varepsilon_2$ by (II). In the Steinhaus case we bound the square of the real part by the squared modulus before applying orthogonality, giving the equation $n_1n_4=n_2n_3$. Thus, for $l_1\neq l_2$ and any $l$, we have
\begin{equation*}
    \mathbb E\left\lvert\widetilde{\mathbb E}
    \widetilde S_{N_{l_1}}\widetilde S_{N_{l_2}}\right\rvert^2
    +\mathbb E\left\lvert\widetilde{\mathbb E}
    \widetilde S_{N_l}^2-1\right\rvert^2
    \ll h(X)^{-2}+\varepsilon_2.
\end{equation*}

Applying Markov's inequality and taking a union bound over all the chosen indices, we obtain that, with probability
\begin{equation*}
    1-O\left(k^2(\log k)^4\left(h(X)^{-2}+\varepsilon_2\right)\right)
\end{equation*}
over the $f(p)$ for $p\leq P$, we have
\begin{equation*}
    \left\lvert\widetilde{\mathbb E}\widetilde S_{N_l}^2-1\right\rvert
    \leq\frac{1}{2}
    \hspace{0.5cm}\text{and}\hspace{0.5cm}
    \left\lvert\widetilde{\mathbb E}
    \widetilde S_{N_{l_1}}\widetilde S_{N_{l_2}}\right\rvert
    \leq\frac{1}{2(\log k)^2}
\end{equation*}
for all $l$ and all $l_1\neq l_2$. In particular, the absolute conditional correlations are at most $(\log k)^{-2}$, also after restricting to the good indices.

On this event and the events used in \eqref{eq:shortIntsNormal}, take $t=\sqrt{\log k}$. Since $p_t\asymp(k\log k)^{-1/2}$ and $\widetilde{\mathbb E}\widetilde S_{N_l}^2\geq 1/2$, we obtain
\begin{equation}\label{eq:shortIntsNormal2}
    \widetilde{\mathbb P}\left(
    \max_{l\ {\rm good}}\widetilde S_{N_l}\leq\sqrt{\frac{\log k}{2}}
    \right)
    \ll \frac{\sqrt{\log k}}{\sqrt k}
    +\frac{1}{\log k}
    +k(\log k)^{\frac{3}{2}}(\varepsilon_1'+\varepsilon_2)^{\frac{1}{30}}.
\end{equation}
Recall that $h(X)=(X/H(X))^\delta$, $k=\lfloor(X/H(X))^{\varepsilon_0}\rfloor$, and $\varepsilon_1'+\varepsilon_2$ is bounded by a fixed negative power of $X$. Taking $\varepsilon_0>0$ sufficiently small compared with $\delta$ and this power, the right-hand side of \eqref{eq:shortIntsNormal2}, together with
\begin{equation*}
    k^2(\log k)^4\left(h(X)^{-2}+\varepsilon_2\right)
    +k\sqrt{\varepsilon_1'+\varepsilon_2},
\end{equation*}
is $\ll 1/\log k$. The last term accounts for the simultaneous fourth moment bound used in \eqref{eq:shortIntsNormal}. Since $\log k\gg\log\log X$, averaging over the $f(p)$ for $p\leq P$ and recalling Lemma \ref{lemma:ignoresmoothShort}, we conclude that
\begin{equation*}
    \max_{l\ {\rm good}}\widetilde S_{N_l}\gg\sqrt{\log k}
\end{equation*}
with probability $1-O((\log\log X)^{-\frac{2}{3}})$.

Finally, we return to the full sums using \eqref{eq:splitting1} and \eqref{eq:splitting2}. For good $l$, the smooth sum divided by $\sqrt{H_l}$ is $\ll(\log\log X)^{\frac{1}{3}}$. When the restriction on $\Omega(n)$ is imposed, the remaining sum is over $n$ with both $P^+(n)>P$ and $\Omega(n)>(1+\varepsilon)\log\log X$. The second moment argument in Lemma \ref{lemma:ignoreTooManyPrimesShort} applies with the additional restriction $P^+(n)>P$, so this sum divided by $\sqrt{H_l}$ is $\ll 1$ for all $l$ with the probability stated there. Since $\sqrt{\lvert\mathcal B_{N_l}\rvert}\asymp\sqrt{H_l}$ uniformly in $l$ in light of Lemmas \ref{lemma:tooManyPrimes} and \ref{lemma:unSmoothSquarefree} and $(\log\log X)^{\frac{1}{3}}=o(\sqrt{\log k})$, we obtain
\begin{equation*}
    \max_l\left\lvert\frac{1}{\sqrt{H_l}}
    \sum_{N_l-H_l<n\leq N_l}f(n)\right\rvert
    \gg\sqrt{\log\frac{X}{H(X)}}
\end{equation*}
with probability $1-O((\log\log X)^{-\frac{2}{3}})$. As $X<N_l\leq Xh(X)\leq X^2$ for all the chosen indices, this finishes the proof of Theorem \ref{thm:localizedBoundShort}.



\bibliographystyle{amsplain}
\bibliography{ref.bib}

@article{Wang-Xu,
  title = {Paucity phenomena for polynomial products},
  volume = {56},
  ISSN = {1469-2120},
  url = {http://dx.doi.org/10.1112/blms.13095},
  DOI = {10.1112/blms.13095},
  number = {8},
  journal = {Bulletin of the London Mathematical Society},
  publisher = {Wiley},
  author = {Wang,  Victor Y. and Xu,  Max W.},
  year = {2024},
  month = may,
  pages = {2718–2726}
}

@article{KlurmanShkredovXu2023,
  author    = {Oleksiy Klurman and Ilya D. Shkredov and Max Wenqiang Xu},
  title     = {On the random {C}howla conjecture},
  journal   = {Geometric and Functional Analysis},
  volume    = {33},
  number    = {3},
  pages     = {749--777},
  year      = {2023},
  doi       = {10.1007/s00039-023-00641-y},
}

@article{Harper2013,
  author    = {Adam J. Harper},
  title     = {On the limit distributions of some sums of a random multiplicative function},
  journal   = {Journal f{\"u}r die reine und angewandte Mathematik},
  volume    = {678},
  pages     = {95--124},
  year      = {2013},
  doi       = {10.1515/crelle-2011-0099},
  eprint    = {arXiv:1012.0207},
}

@article{BobkovGoetze2024,
  author    = {Sergey G. Bobkov and Friedrich G\"otze},
  title     = {Quantified {C}ram\'er-{W}old Continuity Theorem for the {K}antorovich Transport Distance},
  eprint    = {arXiv:2412.10276},
  year      = {2024},
  note      = {Preprint, submitted Dec 13, 2024},
}

@book{HallHeyde1980,
  author    = {Peter Hall and Christopher C. Heyde},
  title     = {Martingale Limit Theory and Its Application},
  publisher = {Academic Press},
  year      = {1980},
  isbn      = {0-12-354550-8},
}

@article{ChinisShala2025,
  author    = {Jake Chinis and Besfort Shala},
  title     = {Random {C}howla’s Conjecture for {R}ademacher Multiplicative Functions},
  journal   = {Transactions of the American Mathematical Society},
  volume    = {378},
  number    = {11},
  year      = {2025},
  doi       = {10.1090/tran/9457},
  eprint    = {arXiv:2409.05952},
}

@article{Hmyrova1966,
  author    = {N. A. Hmyrova},
  title     = {On polynomials with small prime divisors. {II}},
  journal   = {Izvestiya Akademii Nauk SSSR. Seriya Matematicheskaya},
  volume    = {30},
  pages     = {1367--1372},
  year      = {1966},
}

@article{Harper2019_MomentsII,
  author    = {Adam J. Harper},
  title     = {Moments of random multiplicative functions, {II}: {H}igh moments},
  journal   = {Algebra \& Number Theory},
  volume    = {13},
  number    = {10},
  pages     = {2277--2321},
  year      = {2019},
  doi       = {10.2140/ant.2019.13.2277},
  eprint    = {arXiv:1804.04114},
}

@article{Henriot2012_NairTenenbaumUniform,
  author    = {Kevin Henriot},
  title     = {Nair–{T}enenbaum bounds uniform with respect to the discriminant},
  journal   = {Mathematical Proceedings of the Cambridge Philosophical Society},
  volume    = {152},
  number    = {3},
  pages     = {405--424},
  year      = {2012},
  doi       = {10.1017/S0305004111000752},
  eprint    = {arXiv:1102.1643},
}

@article{NairTenenbaum1998,
  author    = {M. Nair and G. Tenenbaum},
  title     = {Short sums of certain arithmetic functions},
  journal   = {Acta Mathematica},
  volume    = {180},
  number    = {1},
  pages     = {119–144},
  year      = {1998},
  doi       = {10.1007/BF02392710},
}

@article{ChatterjeeSoundararajan2012,
  author  = {Sourav Chatterjee and Kannan Soundararajan},
  title   = {Random multiplicative functions in short intervals},
  journal = {Int. Math. Res. Not. IMRN},
  year    = {2012},
  number  = {3},
  pages   = {479--492},
  doi     = {10.1093/imrn/rns041},
}

@article{Pandey2024Squarefree,
  author  = {Mayank Pandey},
  title   = {Squarefree numbers in short intervals},
  journal = {arXiv preprint},
  year    = {2024},
  eprint  = {2401.13981},
  archivePrefix = {arXiv},
  primaryClass = {math.NT}
}

@article{Wintner1944,
  author       = {Wintner, Aurel},
  title        = {Random Factorizations and {R}iemann’s Hypothesis},
  journal      = {Duke Mathematical Journal},
  volume       = {11},
  number       = {2},
  pages        = {267--275},
  year         = {1944},
  note         = {MR 10160}
}

@inproceedings{Halasz1977,
  author    = {Halász, Gábor},
  title     = {On random multiplicative functions},
  booktitle = {Number Theory, Colloquia Mathematica Societatis János Bolyai},
  volume    = {13},
  year      = {1977},
  pages     = {371--378},
  publisher = {North-Holland},
  address   = {Amsterdam}
}

@inproceedings{HalaszRenyi1961,
  author    = {Halász, Gábor and Rényi, Alfréd},
  title     = {On random multiplicative functions},
  booktitle = {Proceedings of the Colloquium on Number Theory},
  year      = {1961},
  pages     = {147--151},
  address   = {Debrecen}
}

@article{gorodetsky_wong_2025_limiting,
  author       = {Gorodetsky, Ofir and Wong, Mo Dick},
  title        = {On the Limiting Distribution of Sums of Random Multiplicative Functions},
  journal      = {arXiv preprint arXiv:2508.12956},
  year         = {2025},
}

@article{Hough2011,
  author    = {Bob Hough},
  title     = {Summation of a random multiplicative function on numbers having few prime factors},
  journal   = {Mathematical Proceedings of the Cambridge Philosophical Society},
  volume    = {150},
  number    = {2},
  pages     = {193--214},
  year      = {2011},
  doi       = {10.1017/S0305004110000514}
}

@article{soundararajan2023clt,
  author    = {Kannan Soundararajan and Max Wenqiang Xu},
  title     = {Central limit theorems for random multiplicative functions},
  journal   = {Journal d’Analyse Mathématique},
  volume    = {151},
  pages     = {343--374},
  year      = {2023},
  doi       = {10.1007/s11854-023-0331-y},
  eprint    = {arXiv:2212.06098},
  archivePrefix = {arXiv},
  primaryClass = {math.NT},
}

@article{PandeyWangXu2024,
  author    = {Mayank Pandey and Victor Y. Wang and Max Wenqiang Xu},
  title     = {Partial sums of typical multiplicative functions over short moving intervals},
  journal   = {Algebra \& Number Theory},
  volume    = {18},
  number    = {2},
  pages     = {389--408},
  year      = {2024},
  doi       = {10.2140/ant.2024.18.389},
}

@article{najnudel2020consecutive,
  author    = {Joseph Najnudel},
  title     = {On consecutive values of random completely multiplicative functions},
  journal   = {Electronic Journal of Probability},
  volume    = {25},
  pages     = {Paper No.\ 59, 28 pp.},
  year      = {2020},
  doi       = {10.1214/20-EJP456}
}

@article{HobanShahIsmailVerreaultZaman2025,
  author    = {Declan Hoban and Jibran Iqbal Shah and Nadya-Catherine Ismail and William Verreault and Asif Zaman},
  title     = {Central limit theorems for random multiplicative functions over function fields},
  journal   = {arXiv preprint},
  eprint    = {2511.22905},
  archivePrefix = {arXiv},
  primaryClass  = {math.NT},
  year      = {2025},
  note      = {arXiv:2511.22905},
  url       = {https://arxiv.org/abs/2511.22905}
}

@article{Harper2023_ASLF,
  author  = {Adam J. Harper},
  title   = {Almost sure large fluctuations of random multiplicative functions},
  journal = {International Mathematics Research Notices},
  year    = {2023},
  number  = {3},
  pages   = {2095--2138},
  doi     = {10.1093/imrn/rnab299},
}

@article{LauTenenbaumWu2013MeanValuesRMF,
  author   = {Lau, Yuk-Kam and Tenenbaum, G{\'e}rald and Wu, Jie},
  title    = {On mean values of random multiplicative functions},
  journal  = {Proceedings of the American Mathematical Society},
  volume   = {141},
  number   = {2},
  pages    = {409--420},
  year     = {2013},
  month    = feb,
  doi      = {10.1090/S0002-9939-2012-11332-2},
  mrnumber = {2996946}
}

@article{Caich2023,
  author  = {Caich, Rachid},
  title   = {Almost sure upper bound for random multiplicative functions},
  journal = {arXiv preprint arXiv:2304.00943},
  year    = {2023},
  url     = {https://arxiv.org/abs/2304.00943}
}

@article{Mastrostefano2022,
  author    = {Daniele Mastrostefano},
  title     = {An almost sure upper bound for random multiplicative functions on integers with a large prime factor},
  journal   = {Electronic Journal of Probability},
  volume    = {27},
  year      = {2022},
  pages     = {Paper No. 32, 21 pp.},
  doi       = {10.1214/22-EJP751},
}

@article{Hardy2025LargePrime,
  author  = {Hardy, Seth},
  title   = {The distribution of partial sums of random multiplicative functions with a large prime factor},
  journal = {arXiv preprint arXiv:2503.06256},
  year    = {2025}, 
}

@article{Hardy2023AlmostSure,
  author    = {Hardy, Seth},
  title     = {Almost sure bounds for a weighted {S}teinhaus random multiplicative function},
  journal   = {Journal of the London Mathematical Society},
  volume    = {110},
  number    = {3},
  pages     = {e12979},
  year      = {2024},
  doi       = {10.1112/jlms.12979},
}

@article{Atherfold2025,
  author    = {Christopher Atherfold},
  title     = {Almost sure bounds for weighted sums of {R}ademacher random multiplicative functions},
  journal   = {arXiv preprint arXiv:2501.11076},
  year      = {2025},
  url       = {https://arxiv.org/abs/2501.11076}
}

@book{Chowla1965,
  author    = {Sarvadaman Chowla},
  title     = {The {R}iemann Hypothesis and {H}ilbert's Tenth Problem},
  publisher = {Gordon and Breach},
  year      = {1965}
}

@article{Elliott1992,
  author  = {Peter D. T. A. Elliott},
  title   = {On the Correlation of Multiplicative Functions},
  journal = {Notas de la Sociedad Matemática de Chile},
  volume  = {11},
  number  = {1},
  year    = {1992},
  pages   = {1--11}
}

@article{Elliott1994,
  author   = {Elliott, P. D. T. A.},
  title    = {On the correlation of multiplicative and the sum of additive arithmetic functions},
  journal  = {Memoirs of the American Mathematical Society},
  year     = {1994},
  volume   = {112},
  number   = {538},
  pages    = {viii+88},
  doi      = {10.1090/memo/0538},
  mrnumber = {1222182}
}

@article{Martin2002,
  author    = {Greg Martin},
  title     = {An asymptotic formula for the number of smooth values of a polynomial},
  journal   = {Journal of Number Theory},
  volume    = {93},
  number    = {2},
  pages     = {108--182},
  year      = {2002},
  doi       = {10.1006/jnth.2001.2722}
}

@article{Harper2013ZetaMax,
  author    = {Adam J. Harper},
  title     = {A note on the maximum of the {R}iemann zeta function, and log-correlated random variables},
  journal   = {arXiv preprint},
  eprint    = {1304.0677},
  archivePrefix = {arXiv},
  primaryClass  = {math.NT},
  year      = {2013},
  url       = {https://arxiv.org/abs/1304.0677}
}

@article{Xu2024btran,
  title        = {Better than square-root cancellation for random multiplicative functions},
  author       = {Xu, Max Wenqiang},
  journal      = {Transactions of the American Mathematical Society, Series B},
  volume       = {11},
  pages        = {482--507},
  year         = {2024},
  doi          = {10.1090/btran/175},
  url          = {https://doi.org/10.1090/btran/175},
}

@article{Hardy2024imrn,
  author       = {Hardy, Seth},
  title        = {Bounds for Exponential Sums With Random Multiplicative Coefficients},
  journal      = {International Mathematics Research Notices},
  year         = {2024},
  number       = {22},
  pages        = {14138--14156},
  doi          = {10.1093/imrn/rnae234},
  url          = {https://doi.org/10.1093/imrn/rnae234},
}

@article{BenatarNishryRodgers2022,
  author       = {Benatar, Jacques and Nishry, Alon and Rodgers, Brad},
  title        = {Moments of Polynomials with Random Multiplicative Coefficients},
  journal      = {Mathematika},
  volume       = {68},
  number       = {1},
  pages        = {191--216},
  year         = {2022},
  doi          = {10.1112/mtk.12121},
  url          = {https://doi.org/10.1112/mtk.12121},
}

@article{Caich2024randomshort,
  title        = {Random multiplicative functions and typical size of character in short intervals},
  author       = {Caich, Rachid},
  journal      = {arXiv:2402.06426},
  year         = {2024},
  eprint       = {2402.06426},
  archivePrefix= {arXiv},
  primaryClass = {math.NT},
  url          = {https://arxiv.org/abs/2402.06426},
}

@article{KowalskiUntrau2025Wasserstein,
  author    = {Emmanuel Kowalski and Th{\'e}o Untrau},
  title     = {Wasserstein metrics and quantitative equidistribution of exponential sums over finite fields},
  journal   = {arXiv preprint},
  eprint    = {2505.22059},
  archivePrefix = {arXiv},
  primaryClass  = {math.NT},
  year      = {2025},
  url       = {https://arxiv.org/abs/2505.22059}
}

@article{Humphries2025QuantitativeEquidistribution,
  author    = {Peter Humphries},
  title     = {Quantitative Equidistribution on Hyperbolic Surfaces and Arithmetic Applications},
  journal   = {arXiv preprint},
  eprint    = {2512.15664},
  archivePrefix = {arXiv},
  primaryClass  = {math.NT},
  year      = {2025},
  url       = {https://arxiv.org/abs/2512.15664}
}

@book{Billingsley1999Convergence,
  author    = {Billingsley, Patrick},
  title     = {Convergence of Probability Measures},
  edition   = {2nd},
  publisher = {John Wiley \& Sons},
  address   = {New York},
  year      = {1999}
}

@article{saksman2020riemann,
  title={The {R}iemann zeta function and {G}aussian multiplicative chaos: Statistics on the critical line},
  author={Saksman, Eero and Webb, Christian},
  journal={Annals of Probability},
  volume={48},
  number={6},
  pages={2680--2754},
  year={2020},
  publisher={Institute of Mathematical Statistics},
  doi={10.1214/20-AOP1433}
}

@article{aymone2021partial,
  title={Partial sums of random multiplicative functions and extreme values of a model for the {R}iemann zeta function},
  author={Aymone, Marco and Heap, Winston and Zhao, Jing},
  journal={Journal of the London Mathematical Society},
  volume={103},
  number={4},
  pages={1618--1642},
  year={2021},
  publisher={Wiley},
  doi={10.1112/jlms.12421}
}

@article{ErdosRenyi1970,
  author  = {Erd{\H{o}}s, Paul and R{\'e}nyi, Alfr{\'e}d},
  title   = {On a new law of large numbers},
  journal = {Journal d'Analyse Math{\'e}matique},
  volume  = {23},
  year    = {1970},
  pages   = {103--111}
}

@article{Harper2020_MomentsRandomMultiplicativeI,
  author       = {Harper, Adam J.},
  title        = {Moments of random multiplicative functions, {I}: {L}ow moments, better than square-root cancellation, and critical multiplicative chaos},
  journal      = {Forum of Mathematics, Pi},
  volume       = {8},
  pages        = {e1},
  year         = {2020},
  doi          = {10.1017/fmp.2019.7},
  url          = {https://doi.org/10.1017/fmp.2019.7}
}

@article{GorodetskyMangerelRodgers2023,
  author  = {Gorodetsky, Ofir and Mangerel, Alexander and Rodgers, Brad},
  title   = {Squarefrees are {Gaussian} in short intervals},
  journal = {J. Reine Angew. Math.},
  volume  = {795},
  year    = {2023},
  pages   = {1--44},
  doi     = {10.1515/crelle-2022-0066}
}

@misc{DurkanPearceCrump2026,
  author        = {Durkan, Benjamin and Pearce-Crump, Andrew},
  title         = {A sharp almost sure upper bound for partial sums of
                   random multiplicative functions},
  year          = {2026},
  eprint        = {2607.29429},
  archivePrefix = {arXiv},
  primaryClass  = {math.NT},
  doi           = {10.48550/arXiv.2607.29429},
  url           = {https://arxiv.org/abs/2607.29429}
}

@misc{Verreault2026,
  author        = {Verreault, William},
  title         = {Almost sure upper bound for sums of random
                   multiplicative functions and critical chaos},
  year          = {2026},
  eprint        = {2608.21354},
  archivePrefix = {arXiv},
  primaryClass  = {math.NT},
  doi           = {10.48550/arXiv.2608.21354},
  url           = {https://arxiv.org/abs/2608.21354}
}

@misc{Atherfold2026,
  author        = {Atherfold, Christopher},
  title         = {Two almost sure upper bounds for sums of random
                   multiplicative functions},
  year          = {2026},
  eprint        = {2611.24847},
  archivePrefix = {arXiv},
  primaryClass  = {math.NT},
  doi           = {10.48550/arXiv.2611.24847},
  url           = {https://arxiv.org/abs/2611.24847}
}

@misc{deLaBretecheWangXu2026,
  author        = {de la Bret{\`e}che, R{\'e}gis and Wang, Victor Y.
                   and Xu, Max Wenqiang},
  title         = {Random Multiplicative Functions and Making Squares
                   from Polynomial Values},
  year          = {2026},
  eprint        = {2607.06398},
  archivePrefix = {arXiv},
  primaryClass  = {math.NT},
  doi           = {10.48550/arXiv.2607.06398},
  url           = {https://arxiv.org/abs/2607.06398}
}

@misc{Pylaiev2026,
  author        = {Pylaiev, Ihor},
  title         = {Random {Chowla} and {Goldbach} Conjectures},
  year          = {2026},
  eprint        = {2611.14579},
  archivePrefix = {arXiv},
  primaryClass  = {math.NT},
  doi           = {10.48550/arXiv.2611.14579},
  url           = {https://arxiv.org/abs/2611.14579}
}

@article{BookerBrowning2016,
  author        = {Booker, Andrew R. and Browning, T. D.},
  title         = {Square-free values of reducible polynomials},
  journal       = {Discrete Analysis},
  year          = {2016},
  volume        = {2016},
  number        = {8},
  pages         = {1--18},
  doi           = {10.19086/da.732},
  eprint        = {1511.00601},
  archivePrefix = {arXiv},
  primaryClass  = {math.NT},
  url           = {https://doi.org/10.19086/da.732}
}

@article{Granville1998,
  author  = {Granville, Andrew},
  title   = {{ABC} allows us to count squarefrees},
  journal = {International Mathematics Research Notices},
  year    = {1998},
  volume  = {1998},
  number  = {19},
  pages   = {991--1009},
  doi     = {10.1155/S1073792898000592},
  url     = {https://doi.org/10.1155/S1073792898000592}
}

@misc{HarperSoundararajanXu2026,
  author        = {Harper, Adam J. and Soundararajan, Kannan
                   and Xu, Max Wenqiang},
  title         = {Distribution of random multiplicative functions
                   in short intervals, with proper normalization},
  year          = {2026},
  eprint        = {2606.29040},
  archivePrefix = {arXiv},
  primaryClass  = {math.NT},
  doi           = {10.48550/arXiv.2606.29040},
  url           = {https://arxiv.org/abs/2606.29040}
}

\end{document}